\documentclass[a4paper, 11pt, reqno]{amsart}
\usepackage[english]{babel}
\usepackage[utf8]{inputenc}
\usepackage[T1]{fontenc}
\usepackage{amsthm, amsmath, amssymb, amsrefs, enumerate, enumitem, mathtools, mathrsfs, lmodern, microtype, xurl}
\usepackage{xcolor}

\usepackage{a4wide}

\usepackage{hyperref}

\theoremstyle{plain}
\newtheorem{thm}{Theorem}[section]
\newtheorem{prop}[thm]{Proposition}
\newtheorem{lemma}[thm]{Lemma}
\newtheorem{cor}[thm]{Corollary}
\theoremstyle{definition}
\newtheorem{defn}[thm]{Definition}

\theoremstyle{remark}
\newtheorem*{rmk}{Remark}
\newtheorem*{rmks}{Remarks}

\newcommand{\C}{\mathbb{C}}
\newcommand{\Hb}{\mathbb{H}}
\newcommand{\Z}{\mathbb{Z}}
\newcommand{\Q}{\mathbb{Q}}
\newcommand{\N}{\mathbb{N}}
\newcommand{\R}{\mathbb{R}}
\newcommand{\slz}{{\text {\rm SL}}_2(\mathbb{Z})}
\newcommand{\re}{\textnormal{Re}}

\newcommand{\vt}[1]{\left\lvert #1 \right\rvert}
\newcommand{\sgn}{\operatorname{sgn}}
\newcommand{\tr}{\operatorname{tr}}
\newcommand{\CT}{\operatorname{CT}}
\newcommand{\Iso}{\operatorname{Iso}}
\newcommand{\reg}{\textnormal{reg}}
\newcommand{\erfc}{\operatorname{erfc}}

\newcommand{\dm}{\mathrm{d}}
\newcommand{\Ec}{\mathcal{E}}
\newcommand{\Fc}{\mathcal{F}}
\newcommand{\Gc}{\mathcal{G}}
\newcommand{\Hc}{\mathcal{H}}
\newcommand{\af}{\mathfrak{a}}
\newcommand{\bff}{\mathfrak{b}}
\newcommand{\cf}{\mathfrak{c}}
\newcommand{\Cf}{\mathfrak{C}}
\newcommand{\ef}{\mathfrak{e}}
\newcommand{\ff}{\mathfrak{f}}
\newcommand{\lf}{\mathfrak{l}}
\newcommand{\tf}{\mathfrak{t}}
\newcommand{\sfr}{\mathfrak{s}}
\newcommand{\Hs}{\mathscr{H}}
\newcommand{\Ks}{\mathscr{K}}
\newcommand{\id}{\mathrm{id}}

\numberwithin{equation}{section}
\setlist{nosep}
\setlist{noitemsep}

\makeatletter
\let\@@pmod\pmod
\DeclareRobustCommand{\pmod}{\@ifstar\@pmods\@@pmod}
\def\@pmods#1{\mkern4mu({\operator@font mod}\mkern 6mu#1)}
\makeatother

\title{A modular framework for generalized Hurwitz class numbers II}
\author{Olivia Beckwith}
\address{Department of Mathematics, Tulane University, New Orleans, LA 70118}
\email{obeckwith@tulane.edu}
\author{Andreas Mono}
\address{Department of Mathematics, Vanderbilt University, 1326 Stevenson Center, Nashville, TN 37240, USA}
\email{andreas.mono@vanderbilt.edu}
\date{\today}

\begin{document}

\begin{abstract}
In a recent preprint, we constructed a sesquiharmonic Maass form $\Gc$ of weight $\frac{1}{2}$ and level $4N$ with $N$ odd and squarefree. Extending seminal work by Duke, Imamo\={g}lu, and T\'{o}th, $\Gc$ maps to Zagier's non-holomorphic Eisenstein series and a linear combination of Pei and Wang's generalized Cohen--Eisenstein series under the Bruinier--Funke operator $\xi_{\frac{1}{2}}$. In this paper, we realize $\Gc$ as the output of a regularized Siegel theta lift of $1$ whenever $N=p$ is an odd prime building on more general work by Bruinier, Funke and Imamo\={g}lu. In addition, we supply the computation of the square-indexed Fourier coefficients of $\Gc$. This yields explicit identities between the Fourier coefficients of $\Gc$ and all quadratic traces of $1$. Furthermore, we evaluate the Millson theta lift of $1$ and consider spectral deformations of $1$.
\end{abstract}



\thanks{The first author was partially supported by Simons Foundation Collaboration Grant \#953473 and National Science Foundation Grant DMS-2401356.}

\maketitle

\section{Introduction and statement of results}

\subsection{Overall motivation}
Half integral weight modular forms arise naturally from Dedekind's $\eta$-function (``$\eta$-quotients'') or from theta functions of quadratic forms with an odd number of variables such as Jacobi's theta function 
\begin{align} \label{eq:thetadef}
\theta(\tau) \coloneqq \sum_{n \in \Z} q^{n^2}, \qquad q \coloneqq e^{2\pi i \tau}, \qquad \tau \in \Hb \coloneqq \left\{\tau = u+iv \in \C \colon v > 0\right\},
\end{align}
and odd powers thereof. Furthermore, it turns out that half integral forms can be constructed from integral weight ones, and vice versa. This goes back to Shimura \cite{shim}, who utilized modular $L$-series to associate an integral weight modular form to every half integral weight form. The Shimura map, and the converse Shintani map \cite{shin}, have found important applications in the Waldspurger formula \cites{Gross-Kohnen-Zagier, koza81, koza84, kohnen85, wald1,wald2}. Subsequently, these maps were embedded into the broader framework of CM-traces. For example, work of Zagier \cite{zagier02} shows that CM traces of modular functions are weight $\frac{3}{2}$ modular forms. More specifically, if $f$ is a modular function for $\slz$ and $D < 0$ is a discriminant, we define the discriminant $D$ trace of $f$ as
\begin{align*}
\tr_D(f) \coloneqq \sum_{Q \in \mathcal{Q}_D \slash {\text {\rm PSL}}_2(\mathbb{Z})} \frac{f\left(z_Q\right)}{\vt{{\text {\rm PSL}}_2(\mathbb{Z})_Q}}, \qquad z_Q \coloneqq \frac{-b + \sgn(a) i\sqrt{\vt{D}}}{2a},
\end{align*}
where $\mathcal{Q}_d$ is the set of integral binary quadratic forms with discriminant $d \in \Z$, and $Q(x,y) = ax^2 + bxy + cy^2 \in \mathcal{Q}_D$. Letting $j$ be Klein's $j$-invariant, Zagier \cite{zagier02} proved that
\begin{align*}
-q^{-1} + 2 + \sum_{D < 0} \tr_D(j - 744) q^{-D}
\end{align*}
is a weakly holomorphic weight $\frac{3}{2}$ modular form on $\Gamma_0(4)$. Other examples of modular CM trace generating functions include a formula for the partition function $p(n)$ by work of Bruinier and Ono \cite{bruono13} and a formula for Andrews' $\operatorname{spt}(n)$ function \cite{andrews} by Ahlgren and Andersen \cite{ahland}. 

Many of these results can be established and extended using theta lifts. The general idea is to compute a (regularized) Petersson inner product of a two-variable theta function and an integral weight modular form. Then, one shows that the Fourier coefficients involve the CM-traces of the integral weight form. Niwa \cite{niwa} re-established Shimura's map along these lines, while Bruinier and Funke \cite{brufu06} embedded Zagier's aforementioned result into the framework of theta lifts and gave vector-valued extensions for arbitrary congruence subgroups. Katok and Sarnak \cite{katoksarnak} provided an analogous construction for Maass cusp forms. 

Choosing the modular function $1$ instead of $j-744$, Funke \cite{funke} re-interpreted \emph{Zagier's non-holomorphic Eisenstein series} \cites{zagier75, zagier76}
\begin{align} \label{eq:Hcdef}
\Hc(\tau) \coloneqq -\frac{1}{12} + \sum_{n \geq 1} \tr_{-n}(1) q^n + \frac{1}{8\pi\sqrt{v}} + \frac{1}{4\sqrt{\pi}} \sum_{n \geq 1} n \Gamma\left(-\frac{1}{2},4\pi n^2 v\right) q^{-n^2},
\end{align} 
as a theta lift, which in turn re-established modularity of $\Hc$. Here and throughout, $\Gamma(s,z)$ denotes the principal branch of the incomplete $\Gamma$-function. The form $\Hc$ demonstrates that it is necessary to consider non-holomorphic modular forms to understand modular properties of the imaginary quadratic traces of $1$. Duke, Imamo\={g}lu, and T\'{o}th \cite{dit11annals} extended this observation to real quadratic traces. More precisely, they proved that the holomorphic generating functions of the real quadratic traces of $j(\tau) - 744$ and $1$ are not modular but are holomorphic parts of so-called ``harmonic Maass forms'' and ``sesquiharmonic Maass forms'' (both defined in Section \ref{sec:prelim}). Subsequent work by Bruinier, Funke, and Imamo\={g}lu \cite{brufuim} connected these results to a regularized Siegel theta lift. In this paper, we combine the case of input function $1$ (the ``sesquiharmonic case'') in Bruinier, Funke, and Imamo\={g}lu's work with our previous work \cite{bemo} to prove explicit formulas for various $L$-values arising as coefficients of such Maass forms.

\subsection{Previous results}
In a previous paper \cite{bemo}, we constructed the first example of a higher level sesquiharmonic Maass form. Such forms appeared first in the aforementioned paper by Duke, Imamo\={g}lu and T\'{o}th \cite{dit11annals} and have seen numerous applications since then \cites{alansa, air, abs, anlagrho, brika20, brdira, jkk14, jkk24, jkkm24, lagrho, mat19, mat20} (including some very recent ones). Duke, Imamo\={g}lu and T\'{o}th's construction is based on the Maass--Eisenstein series
\begin{align}  \label{eq:Fcdef}
\begin{split}
\Fc_{k,N}(\tau,s) &\coloneqq \sum_{\gamma \in \left(\Gamma_0(N)\right)_{\infty} \backslash \Gamma_0(N)} v^{s-\frac{k}{2}} \Big\vert_k\gamma, \quad k \in -\frac{1}{2}\N + 1, \quad \re(s) > 1-\frac{k}{2}, \quad N \in \N, \\
\Fc_{k,4N}^+(\tau,s) &\coloneqq \mathrm{pr}^+\Fc_{k,4N}(\tau,s), \qquad k \in -\N + \frac{3}{2},
\end{split}
\end{align}
in weight $\frac{1}{2}$ and level $4$ projected to Kohnen's plus space. Explicitly, their form is given by
\begin{align} \label{eq:Ecdef}
\Ec(\tau) \coloneqq \frac{1}{3}\frac{\partial}{\partial s} \left(s-\frac{3}{4}\right)\Fc_{\frac{1}{2},4}^+(\tau,s)\Big\vert_{s=\frac{3}{4}}
\end{align}
and it turns out that $\Ec$ is a preimage of Zagier's non-holomorphic Eisenstein series $\Hc$ (see equation \eqref{eq:Hcdef}) under the Bruinier--Funke \cite{brufu02} differential operator 
\begin{align*}
\xi_{\frac{1}{2}} \coloneqq 2iv^{\frac{1}{2}}\overline{\frac{\partial}{\partial\overline{\tau}}}.
\end{align*}
In today's terminology, $\Hc$ is a harmonic Maass form of weight $\frac{3}{2}$ and level $4$, which motivates the term ``sesquiharmonic" for the form $\Ec$.

We let $N$ be an odd and square-free level. Paralleling Duke, Imamo\={g}lu and T\'{o}th's construction, we defined and investigated the function
\begin{align} \label{eq:Gcdef}
\Gc(\tau) \coloneqq \frac{\partial}{\partial s} \left(s-\frac{3}{4}\right)\Fc_{\frac{1}{2},4N}^+(\tau,s) \Big\vert_{s=\frac{3}{4}}. 
\end{align}
To describe its significance briefly, we let
\begin{align} \label{eq:CEdef}
\Hs_{\ell,N}(\tau) \coloneqq \sum_{n \geq 0} H_{\ell,N}(n) q^n, \qquad \ell \mid N,
\end{align}
be Pei and Wang's \emph{generalized Cohen--Eisenstein series} \cite{peiwang}, where $H_{\ell,N}(n)$ are the generalized Hurwitz class numbers defined in equations \eqref{eq:HellNdef} and \eqref{eq:HNNdef} below. If $N=1$, we have that $H_{1,1}(n) = \tr_{-n}(1)$ are the classical Hurwitz class numbers. If $N$, $\ell > 1$ then $\Hs_{\ell,N}$ is a modular form of weight $\frac{3}{2}$ and level $4N$ in the plus subspace. Altogether, they generate the Eisenstein plus subspace $E_{\frac{3}{2}}^+(4N)$, while the form $\Hs_{1,N}$ fails to be modular. These series generalize classical work by Cohen \cite{cohen75} in the level aspect, and they appeared in \cites{mmrw, luozhou} very recently. This established, the main result of our previous paper \cite{bemo}*{Theorem 1.3 (iii)} asserts that $\Gc$ is a sesquiharmonic Maass form of weight $\frac{1}{2}$ and level $4N$ satisfying
\begin{align} \label{eq:GcShadow}
\frac{1}{4}\xi_{\frac{1}{2}} \Gc(\tau) = \Big(\prod_{p\mid N} \frac{1}{p+1}\Big) \left(\Hs_{1,1}(\tau) - \Hc(\tau)\right) - \frac{1}{N} \sum_{\ell \mid N} \ell \Big(\prod_{p \mid \ell} \frac{1}{1-p}\Big) \Hs_{\ell,N}(\tau).
\end{align}
In particular, this implies that a certain linear combination of $H_{1,1}(n)$ and $H_{1,N}(n)$ generates a modular form of weight $\frac{3}{2}$ and level $4N$ for $N > 1$ odd and square-free, see \cite{bemo}*{Theorem 1.1}. Furthermore, this gives rise to a higher level analog of Zagier's Eisenstein series from equation \eqref{eq:Hcdef} that involves Pei and Wang's numbers $H_{1,N}(n)$ instead of the classical Hurwitz class numbers $H_{1,1}(n)$, see \cite{bemo}*{Theorem 1.2}.

\subsection{Main results}
Among the various theta lifts in the literature, the aforementioned regularized Siegel theta lift $I^{\reg}(\tau,f)$ of a weight $0$ weak Maass form $f$ to a form of weight $\frac{1}{2}$ (both defined in Section \ref{sec:prelim}) has attracted a lot of attention over the past decades as well as very recently, see \cites{bor98, bor00, bruinierhabil, brufu02, BEY, anbs, BS, males, mamo, brufuim, alfesthesis, schw18, crfu, hoevel, BO, BY} to name a few. The regularization goes back to an idea by Borcherds \cite{bor98} and Harvey--Moore \cite{hamo} and employs a truncated fundamental domain and meromorphic continuation. Following \cite{brufuim}*{(5.4)}, we denote  the component functions of the vector-valued Siegel theta lift by $I_h^{\reg}(\tau,f)$ throughout. Then, we employ the projection map to the plus space from \eqref{eq:plusspaceprojection} and have the following result.

\begin{thm} \label{thm:GcAsThetaLift}
Let $N=p$ be an odd prime.
\begin{enumerate}[label={\rm (\roman*)}]
\item The vector space of weight $\frac{1}{2}$ sesquiharmonic Maass forms on $\Gamma_0(4p)$ that are of at most polynomial growth towards all cusps and satisfy the plus space condition is three-dimensional.
\item Let $I_h^{\reg}(\tau,1)$ be as in \cite{brufuim}*{Theorem 4.2} (see Lemma \ref{lem:bfimain2} below). Let $\gamma$ be the Euler--Mascheroni constant. Then, we have
\begin{align*}
\frac{2}{p-1}\Ec(\tau) + \Gc(\tau) = \Cf(p) \theta(\tau) + \frac{\sqrt{p}}{p^2-1}\sum_{h \pmod*{2p}} I_h^{\reg}(4p\tau,1),
\end{align*}
where 
\begin{align} \label{eq:frakCdef}
\Cf(p) \coloneqq  \frac{4p}{\pi\left(p^2-1\right)} \left[\gamma - \log(2) - \frac{\zeta'(2)}{\zeta(2)} - \frac{p\log(p)}{2(p+1)} + \frac{1}{4}\left(\log(\pi)-\gamma\right) \right].
\end{align}
\end{enumerate}
\end{thm}

\begin{rmks}
\
\begin{enumerate}[label={\rm (\arabic*)}]
\item Ahlgren, Andersen and Samart \cite{alansa}*{Section 6} point out a few corrections to \cite{brufuim}*{Theorem 4.2}, which we incorporated in Lemma \ref{lem:bfimain2} below.
\item The functions $\Ec$ and $\theta$ admit representations as Siegel theta lifts as well.
\item In the first part, a basis is given by $\theta$, $\Ec$, and $\Gc$.
\item A corollary is that there are no weight $\frac{1}{2}$ non-holomorphic harmonic Maass forms of moderate growth satisfying the plus-space condition, see Proposition \ref{prop:sesquiharmonicequality}.
\end{enumerate}
\end{rmks}

Theorem \ref{thm:GcAsThetaLift} yields a higher level analog of a result by Bruinier and Funke \cite{brufu06}*{Theorem 7.1} straightforwardly. To this end, let $J^{\reg}(\tau, f)$ be the regularized Kudla--Millson theta lift \cites{kumi, kumi2} of a weight $0$ weak Maass form $f$ to a form of weight $\frac{3}{2}$ (using the same regularization as for $I^{\reg}$, see \cite{brufu06}*{(4.6)}). Analogous to the regularized Siegel theta lift, the notation $J_h^{\reg}(\tau, f)$ refers to the individual components of the vector-valued theta lift $J^{\reg}$. We recall Pei and Wang's generalized Cohen--Eisenstein series $\Hs_{p,p}$ from \eqref{eq:CEdef}, which generates the Eisenstein plus subspace $E_{\frac{3}{2}}^+(4p)$ whenever $p$ is an odd prime (see \cite{peiwang}*{Theorem 1}).
\begin{thm} \label{thm:millsonthetalift}
Let $N=p$ be an odd prime. Then, we have
\begin{align*}
\sum_{h \pmod*{2p}} J_h^{\reg}(4p\tau,1) = \frac{2}{\pi \sqrt{p}} \Hc(\tau) - \frac{1}{\pi \sqrt{p}}\Hs_{p,p}(\tau).
\end{align*}
\end{thm}

\begin{rmk}
The functions $\Hs_{p,p}$, $\Hs_{1,1}$ and $\Hs_{1,p}$ satisfy a linear relation by \cite{bemo}*{Theorem 1.1}.
\end{rmk}

Next, we would like to relate the Fourier coefficients of the forms appearing in Theorem \ref{thm:GcAsThetaLift} (ii) to each other. However, the square-indexed Fourier coefficients of $\Ec$ were not computed in \cite{dit11annals} as they correspond to poles of certain Kloosterman zeta functions (see \cite{bemo}*{Proposition 3.7}). Nevertheless, they attracted a lot of attention recently \cites{an15, alansa}. In \cite{bemo}*{Section 3}, we obtained an explicit description of all the Fourier coefficients of $\Gc$, although we did not carry out the calculation of the positive square-indexed coefficients. Here, we supplement the computation of the square-indexed Fourier coefficients $\bff(n)$ and $\cf(n)$ of $\Ec$ and $\Gc$ respectively. The full Fourier expansion of both forms can be found in Lemmas \ref{lem:EcExpansion}, \ref{lem:GcExpansion} below. The constant coefficient $\bff(0)$ of $\Ec$ was computed in \cite{alansa}*{(1.7)}.
\begin{thm} \label{thm:cfsquareindices}
Let $N=p$ be an odd prime, $\tf_{N}$ be as in equation \eqref{eq:t(n)def}, and $m \in \N$.
\begin{enumerate}[label={\rm (\roman*)}]
\item We have
\begin{align*}
\bff\left(m^2\right) =
\frac{2}{3\pi} \left(\gamma - 2\frac{\zeta'(2)}{\zeta(2)} - \log(4) + \frac{1}{2}\log(\pi) - \tf_{1}(m)\right).
\end{align*}
\item We have
\begin{align*}
\frac{2}{3}(1-i) \pi \cf(0) = \frac{2}{\pi(p+1)} \left(\gamma - \log(2) - \frac{\zeta'(2)}{\zeta(2)} - \frac{p^2\log(p)}{p^2-1} \right).
\end{align*}
\item We have
\begin{multline*}
\hspace*{\leftmargini} \frac{2}{3}(1-i) \pi \cf\left(m^2\right) = \frac{2}{\pi(p+1)} \left[ \gamma - 2 \frac{\zeta'(2)}{\zeta(2)} +
\log(2) \left(2^{-\nu_2(m)}-3\right) + \right. \\ \left.
\log(p) \left(\frac{1}{p+1} + \frac{p+1}{p-1}p^{-\nu_p(m)} - \frac{2p}{p-1} \right) - \tf_{4p}(m) \right].
\end{multline*}
\end{enumerate}
\end{thm}

\begin{rmk}
Let $p$ be an odd prime. The coefficient $\cf(0)$ appears in the Kronecker limit formula for the weight $0$ and level $p$ non-holomorphic Eisenstein series $\Fc_{0,p}(\tau,s)$ from equation \eqref{eq:Fcdef}. This goes back to a result of Goldstein \cite{goldstein}*{Theorems 3-1 and 3-3}, which Vassileva \cite{vass}*{Theorem 3.3.1} specialized to the case of $\Gamma_0(p)$. Explicitly, we have
\begin{align*}
\lim_{s \to 1} \left(\Fc_{0,p}(\tau,s) - \frac{\frac{3}{\pi(p+1)}}{s-1}\right) =  2(1-i)\pi\cf(0) - \frac{6}{\pi (p+1)}\log\left(\sqrt{v} \vt{\frac{\eta(p\tau)^p}{\eta(\tau)}}^{\frac{2}{p-1}} \right),
\end{align*}
where $\eta$ is the Dedekind $\eta$-function. The function on the right hand side is a weight $0$ and level $p$ sesquiharmonic Maass form.
\end{rmk}

Equipped with explicit formulas for all the Fourier coefficients of $\Ec$ and $\Gc$, we relate them to the Fourier coefficients of the theta lift in Theorem \ref{thm:GcAsThetaLift} (ii). From the various non-constant parts of the involved Fourier expansions, we obtain the following results.
\begin{thm} \label{thm:fouriercomparisonmain}
Let $H_{\ell,p}(n)$ be Pei and Wang's aforementioned generalized Hurwitz class numbers (see equations \eqref{eq:HellNdef}, \eqref{eq:HNNdef}), and $\mathcal{Q}_{p,n}$ be as in equation \eqref{eq:Quadraticformssets} below. If $n < 0$ with $n \equiv 0,1 \pmod*{4}$, then we have
\begin{align*}
\hspace*{\leftmargini} \sum_{Q \in \mathcal{Q}_{p,n} \slash \Gamma_0(p)} \frac{1}{\vt{\Gamma_0(p)_Q}} = \frac{4(p+1)}{p} H_{1,p}(-n) - \frac{2(p+1)}{p-1} H_{p,p}(-n).
\end{align*}
\end{thm}

\begin{rmks}
\
\begin{enumerate}[label={\rm (\arabic*)}]
\item We note that the left hand side is the higher level analog of the definition of the classical Hurwitz class numbers in level one. 
\item The traces of positive non-square discriminant evaluate in terms of the fundamental unit and the class number associated to that discriminant as well as local factors pertaining to the given level, see Theorem \ref{thm:fouriercomparisondetail}.
\item The traces of positive square discriminant evaluate in terms of $\tf_{1}$ and $\tf_{4p}$ defined in equation \eqref{eq:t(n)def} among other explicit terms, see Theorem \ref{thm:fouriercomparisondetail}.
\end{enumerate}
\end{rmks}

In addition, Bruinier, Funke and Imamo\={g}lu \cite{brufuim} discuss a second regularization of their Siegel theta lift by employing so-called spectral deformations of the input function, and compare this to the regularization by truncating a fundamental domain. The upshot is that the two regularizations of the theta lift of $1$ differ by a multiple of $\theta$ in our case (see \cite{brufuim}*{Theorem 1.3}), which leads to the following result.
\begin{thm} \label{thm:specdefcor}
Let $N=p$ be an odd prime, let
\begin{align*}
\zeta^*(s) \coloneqq \frac{\Gamma\left(\frac{s}{2}\right)\zeta(s)}{\pi^{\frac{s}{2}}},
\end{align*}
and let $\Cf(p)$ be the constant from \eqref{eq:frakCdef}. Then, we have
\begin{multline*}
\frac{2}{p-1}\Ec(\tau) + \Gc(\tau) = \frac{\pi p}{2(p-1)} \CT_{s=1} \left[\frac{4^{\frac{s}{2}}\zeta^*(s)}{\zeta^*(2s-1)} \Fc_{\frac{1}{2},4p}^+\left(\tau,\frac{s}{2}+\frac{1}{4}\right) \right] \\
+\left[\Cf(p) - \pi^2 p (p+1)\left(\gamma+\log(\pi)-2\frac{\zeta'(2)}{\zeta(2)}\right)\right]  \theta(\tau).
\end{multline*}
\end{thm}

\begin{rmk}
Theorem \ref{thm:specdefcor} is a higher level analog of \cite{brufuim}*{Corollary 6.3} combined with our result from Theorem \ref{thm:GcAsThetaLift} (ii).
\end{rmk}

\subsection{Outline of the paper}
The paper is organized as follows. We summarize the framework of this paper in Section \ref{sec:prelim}. Section \ref{sec:ProofOfThm1.3} is devoted to the explicit computation of the square-indexed Fourier coefficients of the forms $\Ec$ and $\Gc$, which proves Theorem \ref{thm:cfsquareindices}. In particular, the computation of the constant terms of $\Gc$ enables us to prove Theorem \ref{thm:GcAsThetaLift} in Section \ref{sec:ProofOfThm1.1}. The content of Section \ref{sec:ProofOfThm1.2} is the proof of Theorem \ref{thm:millsonthetalift} by evaluating the Millson theta lift of $1$. Section \ref{sec:ProofOfThm1.4} deals with the explicit relations between the Fourier coefficients of $\Ec$, $\Gc$, $\theta$ and the Siegel theta lift of $1$ projected to the plus space, which proves Theorem \ref{thm:fouriercomparisonmain}. Lastly, we utilize spectral deformations of $1$ to prove Theorem \ref{thm:specdefcor} in Section \ref{sec:ProofOfThm1.5}.

\section*{Notation}

\subsection*{General modular forms notation}
\begin{itemize}
\item $\tau = u+iv \in \Hb$, $z = x+iy \in \Hb$, $q = e^{2\pi i \tau}$, $\dm\mu(z) \coloneqq \frac{\dm x \dm y}{y^2}$,
\item $p$ is an odd prime, 
\item $\ell$ is a divisor of $N \in \N$, where $N$ is odd and square-free in the context of $\Hs_{\ell,N}$,
\item $n = tm^2$ is a Fourier coefficient with $t$ a fundamental discriminant,
\item $\gamma$ is either the Euler--Mascheroni constant or a $2 \times 2$ matrix,
\item $\Gamma(s,z)$ is the (analytic continuation of the) principal branch of the incomplete $\Gamma$-function (as a function of $s$), $\Gamma(s) = \Gamma(s,0)$ is the $\Gamma$-function,
\item $\af \in \Q \cup \{i\infty\}$ is a cusp with scaling matrix $\sigma_{\af} \in \slz$,
\item $\Gamma_0(N) = \left\{\left(\begin{smallmatrix} a & b \\ c & d \end{smallmatrix}\right) \in \slz \colon c \equiv 0 \pmod*{N} \right\}$ is the Hecke congruence subgroup of $\slz$,
\item $(\Gamma_0(N))_{\infty} = \left\{\pm \left(\begin{smallmatrix} 1 & n \\ 0 & 1 \end{smallmatrix}\right) \colon n \in \Z \right\} \leq \Gamma_0(N)$ is the stabilizer of the cusp $i\infty$,
\item $\nu_p(n) \in \N_0 \cup \{\infty\}$ is the $p$-adic valuation of $n$,
\item $\left(\frac{a}{b}\right)$ is the Kronecker symbol,
\item $\chi_{d}(n) = \left(\frac{d}{n}\right)$ is a Dirichlet character,
\item $\id = \chi_{1}$ is the principal character of modulus $1$,
\item $\varepsilon_d$ refers to a quantity in the theta multiplier,
\item $\widetilde{\varepsilon}_t$ refers to the fundamental unit of $\Q\left(\sqrt{t}\right)$,
\item $\widehat{\varepsilon}_{\af}$ is a parameter associated to a cusp $\af$, see the paragraph preceeding Lemma \ref{lem:cuspparameters},
\item $L(\chi,s)$ is the (complete) Dirichlet $L$-function, $\zeta(s)$ is the Riemann zeta function, $\zeta^*(s)$ is the completed Riemann $\zeta$-function,
\item $L_N(\chi,s)$ the incomplete Dirichlet $L$-function defined in equation \eqref{eq:L_Ndef},
\item $\mu(n)$ the M{\"o}bius function,
\item $\sigma_{\ell,N,s}(n)$, $\ell \mid N$, and $\sigma_{N,s}(n)$ are certain modifications of the standard sum of powers of divisors function $\sigma_s$, see equation \eqref{eq:divisorsums}
\item $M_{k}(N)$ is the $\C$-vector space of holomorphic modular forms of weight $k$ for $\Gamma_0(N)$,
\item $M_{k}^+(N) \subseteq M_{k}(N)$ is Kohnen's plus subspace inside $M_{k}(N)$,
\item $S_{k}^+(N) \subseteq M_{k}^+(N)$ is the subspace of cusp forms inside $M_{k}^+(N)$,
\item $E_{k}^+(N) \subseteq M_{k}^+(N)$ is the Eisenstein subspace inside $M_{k}^+(N)$,
\item $W_{\mu,\nu}(y)$ is the usual $W$-Whittaker function,
\item $\xi_k = 2iv^k\overline{\frac{\partial}{\partial\overline{\tau}}}$ is the Bruinier--Funke differential operator,
\item $\Delta_k = -\xi_{2-k}\xi_k$ is the hyperbolic Laplace operator, see equation \eqref{eq:Deltadef},
\item $\cdot \vert_{k}\cdot$ resp.\ $\cdot \vert_{k, L}\cdot$ is the Petersson slash operator, see equations \eqref{eq:PeterssonSVdef} and \eqref{eq:PeterssonVVdef},
\item $\cdot\vert\mathrm{pr}^+$ is the projection operator into Kohnen's plus space, see \cite{thebook}*{Proposition 6.7} for example,
\end{itemize}

\subsection*{Quadratic traces and class numbers}
\begin{itemize}
\item $\mathcal{Q}_{N,n}$ and $\mathcal{Q}_{N,n,h}$ denote certain sets of integral binary quadratic forms of discriminant $n$, see equation \eqref{eq:Quadraticformssets},
\item $h(n)$ denotes the class number in the wide sense, 
\item $h^*(n)$ denotes the modified class number, see equation \eqref{eq:h*def},
\item $H_{\ell,N}(n)$, $\ell \mid N$, are the generalized Hurwitz class numbers, see equations \eqref{eq:HellNdef}, \eqref{eq:HNNdef},
\item $H(n) = H_{1,1}(n)$ are the classical Hurwitz class numbers,
\item $\tau_Q \in \Hb$ denotes the Heegner point associated to the integral binary quadratic form $Q$,
\item $S_Q \subseteq \Hb$ denotes the geodesic associated to the integral binary quadratic form $Q$,
\item $\tr_{m,h}(f)$ denotes the imaginary quadratic ($m < 0$) resp.\ real quadratic ($m>0$) trace of the function $f$, see equations \ref{eq:trimagdef} and \ref{eq:trrealdef} including the regularization described in \cite{brufuim}*{Subsection 3.4} whenever $m=\square$,
\item $Z(m,h)$ refers to the Heegner divisor of index $(m,h)$, see \cite{brufuim}*{p.\ 55} for more details,
\item $\mathrm{vol}(M)$ is the hyperbolic volume of $M \subseteq \Hb$ normalized by $-\frac{1}{2\pi}$, see equation \eqref{eq:BFIvol},
\end{itemize}

\subsection*{Scalar-valued forms}
\begin{itemize}
\item $\Fc_{k,4N}(\tau,s)$ is the weight $k \in -\frac{1}{2}\N+1$ and level $4N$ Maass--Eisenstein series with spectral parameter $s \in \C$ satisfying $\re(s) > 1-\frac{k}{2}$, see equation \eqref{eq:Fcdef}
\item $\Fc_{k,4N}^+(\tau,s)$ is the projection of $\Fc_{k,4N}(\tau,s)$ to Kohnen's plus space, see equation \eqref{eq:Gcdef}
\item $\Ec(\tau)$ is Duke, Imamo\={g}lu and T\'{o}th's sesquiharmonic Maass form of weight $\frac{1}{2}$ and level $4$, see equation \eqref{eq:Ecdef},
\item $\Gc(\tau)$ is the higher level analog of $\Ec$, see equation \eqref{eq:Gcdef},
\item $\Hs_{\ell,N}(\tau)$, where $\ell \mid N$, are the generating functions of $H_{\ell,N}(n)$; if $N > 1$ is odd and square-free and $1 < \ell \mid N$ then $\Hs_{\ell,N} \in E_{\frac{3}{2}}^+(4N)$,
\item $\Hc(\tau)$ is Zagier's non-holomorphic Eisenstein series, see equation \eqref{eq:Hcdef},
\item $\theta(\tau) = \sum_{n \in \Z} q^{n^2} \in M_\frac{1}{2}(4)$ is the classical Jacobi theta function, see equation \eqref{eq:thetadef},
\end{itemize}

\subsection*{Vector-valued notation and theta lifts}
\begin{itemize}
\item $V$ is the three-dimensional space of traceless $2\times2$ matrices with rational entries,
\item $L \subseteq V$ is the lattice associated to $\Gamma_0(N)$, see equation \eqref{eq:latticedef}, its dual lattice $L'$ is given by equation \eqref{eq:latticedualdef},
\item $\mathrm{Mp}_2(\Z)$ is the metaplectic double cover of $\slz$, see equation \eqref{eq:Mp2Zdef},
\item $\rho_{L}$ is the Weil representation associated to $L$, see equation \eqref{eq:Weildef},
\item $\ef_{h}$ with $h \in L'\slash L$ is the standard basis of the group ring $\C[L'\slash L]$,
\item $M_{k, L}$ is the space of holomorphic vector-valued modular forms,
\item $\Iso(V)$ is the set of isotropic lines $\lf$ in $V$ with a certain parameters $\alpha_{\lf}$, $\beta_{\lf}$, and $k_{\lf}$ see the paragraph before Lemma \ref{lem:cuspparameters} as well as Lemma \ref{lem:bfimain2},
\item $K_{\lf}$ is an even lattice associated to an isotropic line $\lf \in \Iso(V)$, see equation \eqref{eq:unarylattice},
\item $\varphi_0(X,\tau,z)$ denotes the standard Gaussian as in \cite{brufuim}*{Subsection 5.1},
\item $\Theta_L(\tau,z,\varphi)$ denotes the theta kernel arising from the function $\varphi$. If $\varphi = \varphi_0$ then this is the Siegel theta kernel. If $\varphi = \varphi_{\text{KM}}$ (see \cite{kumi}) then this is the Millson theta kernel. The component functions of $\Theta_L(\tau,z,\varphi)$ are denoted by $\theta_{h}(\tau,z,\varphi)$, see equations \eqref{eq:thetacompdef} and \eqref{eq:Thetakerneldef},
\item $\tilde\Theta_{K_\af}(\tau)$ is as in equation \eqref{eq:Thetatildedef},
\item $P_{0,0}(\tau,s)$ denotes the vector-valued weight $\frac{1}{2}$ Maass--Eisenstein series (see equation \eqref{eq:vvEisDef}) with component functions $P_{0,0}^{(h)}$,
\item $\CT_{s=\delta}(f)$ denotes the constant term in the Laurent expansion of $f$ as a function of $s$ about $\delta$,
\item $\mathbb{F}_T(N)$ refers to a fundamental domain for $\Gamma_0(N)$ truncated at height $T > 0$, see equation \eqref{eq:funddomtrunc}
\item $I^{\reg}(\tau,f)$ refers to the regularized Siegel theta lift from weight $0$ to weight $\frac{1}{2}$ as in Bruinier, Funke and Imamo\={g}lu \cite{brufuim}. Its components are denoted by $I_h^{\reg}(\tau,f)$, see equations \eqref{eq:ThetaLiftDefn} and \eqref{eq:ThetaLiftComponents},
\item $J^{\reg}(\tau,f)$ refers to the regularized Millson theta lift from weight $0$ to weight $\frac{3}{2}$. Its components are denoted by $J_h^{\reg}(\tau,f)$, see equation \eqref{eq:ThetaLiftMillsonDef},
\end{itemize}

\subsection*{Special notation}
\begin{itemize}
\item $\Ks_{0,N}(m,n;s)$ is the usual Kloosterman zeta function associated to the cusp $i\infty$ and with no multiplier defined in equation \eqref{eq:Ks0Ndef},
\item $\widetilde{\Ks}_{0,N}(m,n;s)$ is the modified Kloosterman zeta function associated to the cusp $0$ and with no multiplier defined in equation \eqref{eq:Ks0Ntildedef},
\item $\Ks_{k,4N}^+(m,n;s)$ is the plus space Kloosterman zeta function associated to the cusp $i\infty$ and with the theta multiplier defined in equation \eqref{eq:Ks4Nplusdef},
\item $\alpha(y)$ is Duke, Imamo\={g}lu and T\'{o}th's special function, which we renormalize as in equation \eqref{eq:alphaspecialdef}, 
\item $\Fc(t)$ is Bruinier, Funke and Imamo\={g}lu's special function, which we correct as suggested in \cite{alansa}*{Section 6} and which coincides with $\alpha$ essentially, see equation \eqref{eq:FcSpecialdef} and Lemma \ref{lem:specialfunctionrelation},
\item $\psi(z) = \frac{\Gamma'(z)}{\Gamma(z)}$ is the Euler Digamma function with $\psi(0) \coloneqq - \gamma$,
\item $T_{N,s}^{\chi}$ is a certain divisor-type sum arising from the square part of the Fourier index of $\Fc_{\frac{1}{2},4N}^+$, see equation \eqref{eq:TNschidef},
\item $\tf_{N}$ is the $s$-derivative of $T_{N,s}^{\chi}$ at $s=\frac{3}{4}$ (up to renormalization), see equation \eqref{eq:t(n)def},
\item $\bff(n)$ are certain Fourier coefficients of $\Ec$ determined in Theorem \ref{thm:cfsquareindices} (i),
\item $\cf(n)$ are certain Fourier coefficients of $\Gc$ defined in equation \eqref{eq:c(n)def},
\item $\Cf(p)$ is a certain constant introduced in equation \eqref{eq:frakCdef},
\item $f_s(\tau)$ is the spectral deformation of the constant function $1$ defined in Lemma \ref{lem:spectraldeformations},
\item $B_{\af}(s)$ refers to the multiple of $v^{1-s}$ in the Fourier expansion of $f_s$ about the cusp $\af$, see equation \eqref{eq:Bafdef},
\item $\widetilde{A(2,n)}$ is a local factor, which is a renormalized version of the one in \cite{bemo}*{Proposition 3.5}, see Theorem \ref{thm:fouriercomparisondetail},
\item $A(p,n)$ is the local factor from \cite{bemo}*{Proposition 3.5} (without any modifications), see Theorem \ref{thm:fouriercomparisondetail} as well.
\end{itemize}

\section{Preliminaries} \label{sec:prelim}

\subsection{Holomorphic and non-holomorphic modular forms}

Let $k \in \frac{1}{2}\Z$. Choosing the principal branch of the square-root throughout, the \emph{(Petersson) slash operator} is defined as 
\begin{align} \label{eq:PeterssonSVdef}
\left(f\vert_k\gamma\right)(\tau) \coloneqq \begin{cases}
(c\tau+d)^{-k} f(\gamma\tau) & \text{if } k \in \Z, \\
\left(\frac{c}{d}\right)\varepsilon_d^{2k}(c\tau+d)^{-k} f(\gamma\tau) & \text{if } k \in \frac{1}{2}+\Z,
\end{cases}
\quad \gamma = \left(\begin{matrix} a & b \\ c& d \end{matrix}\right) \in \begin{cases}
\slz & \text{if } k \in \Z, \\
\Gamma_0(4) & k \in \frac{1}{2}+\Z,
\end{cases}
\end{align}
where $\left(\frac{c}{d}\right)$ denotes the Kronecker symbol, and $\varepsilon_d \coloneqq 1,i$ if $d \equiv \pm1 \pmod*{4}$. The \emph{weight $k$ hyperbolic Laplace operator} is defined as
\begin{align} \label{eq:Deltadef}
\Delta_k \coloneqq - \xi_{2-k} \xi_k = -v^2\left(\frac{\partial^2}{\partial u^2}+\frac{\partial^2}{\partial v^2}\right) + ikv\left(\frac{\partial}{\partial u} + i\frac{\partial}{\partial v}\right).
\end{align}
\begin{defn}
Let $k \in \frac{1}{2}\Z$, $N \in \N$, and $f \colon \Hb \to \C$ be a smooth function. Assume that $4 \mid N$ whenever $k \in \frac{1}{2} + \Z$.
\begin{enumerate}[label={\rm (\alph*)}]
\item We say that $f$ is a \emph{(holomorphic) modular form} of weight $k$ on $\Gamma_0(N)$, if $f$ has the following properties:
\begin{enumerate}[label={\rm (\roman*)}]
\item we have $f\vert_k \gamma = f$ for all $\gamma \in \Gamma_0(N)$, 
\item we have $\xi_k f = 0$, that is $f$ is holomorphic on $\Hb$,
\item we have that $f$ is holomorphic at every cusp of $\Gamma_0(N)$.
\end{enumerate}
\item We say that $f$ is a \emph{harmonic (resp.\ weak) Maass form of weight $k$} on $\Gamma_0(N)$, if $f$ has the following properties:
\begin{enumerate}[label={\rm  (\roman*)}]
\item we have $f\vert_k \gamma = f$ for all $\gamma \in \Gamma_0(N)$, 
\item we have $-\xi_{2-k} \xi_k f = 0$ (resp.\ $-\xi_{2-k} \xi_k f = \lambda f$) for all $\tau \in \Hb$,
\item we have that $f$ is of at most linear exponential growth towards all cusps of $\Gamma_0(N)$.
\end{enumerate}
\item We say that $f$ is a \emph{sesquiharmonic Maass form of weight $k$} on $\Gamma_0(N)$, if $f$ has the following properties:
\begin{enumerate}[label={\rm  (\roman*)}]
\item we have $f\vert_k \gamma = f$ for all $\gamma \in \Gamma_0(N)$, 
\item we have $\xi_k \xi_{2-k} \xi_k f = 0$, that is $f$ is sesquiharmonic on $\Hb$,
\item we have that $f$ is of at most linear exponential growth towards all cusps of $\Gamma_0(N)$.
\end{enumerate}
\end{enumerate}
Any of these forms is an element of \emph{Kohnen's plus space} if its weight is half integral and its Fourier coefficients are supported on indices $n$ satisfying $(-1)^{k-\frac{1}{2}} n \equiv 0$, $1 \pmod*{4}$. 
\end{defn}

Let $V$ be the three-dimensional space of rational traceless $2 \times 2$ matrices. The lattice
\begin{align} \label{eq:latticedef}
L \coloneqq \left\{\left(\begin{matrix} b & \frac{c}{N} \\ a & -b \end{matrix}\right) \colon a,b,c \in \Z\right\}
\end{align}
is associated to $\Gamma_0(N)$. Its dual lattice is equal to
\begin{align} \label{eq:latticedualdef}
L' \coloneqq \left\{\left(\begin{matrix} b & \frac{c}{N} \\ a & -b \end{matrix}\right) \colon a,c \in \Z, b \in \frac{1}{2N}\Z\right\}.
\end{align}
We have $L'\slash L\cong \Z\slash2N\Z$ equipped with the quadratic form $x \mapsto x^2 / (4N)$, and the level of $L$ is $4N$. The quadratic form associated to $L$ is $Q(a,b,c) = Nb^2+ac$. The corresponding bilinear form is $\left(\left(\begin{smallmatrix} b_1 & \frac{c_1}{N} \\ a_1 & -b_1 \end{smallmatrix}\right), \left(\begin{smallmatrix} b_2 & \frac{c_2}{N} \\ a_2 & -b_2 \end{smallmatrix}\right)\right) = 2Nb_1b_2+a_1c_2+a_2c_1$, which yields signature $(2,1)$. 

The metaplectic double cover
\begin{align} \label{eq:Mp2Zdef}
\mathrm{Mp}_2(\Z) \coloneqq \left\{ (\gamma, \phi) \colon \gamma = \left(\begin{smallmatrix} a & b \\ c & d \end{smallmatrix}\right)\in \slz, \ \phi\colon \Hb \rightarrow \C \text{ holomorphic}, \ \phi^2(\tau) = c\tau+d  \right\},
\end{align}
of $\slz$ is generated by the pairs
\begin{align*}
\widetilde{T} \coloneqq \left(\left( \begin{matrix} 1 & 1 \\ 0 & 1 \end{matrix} \right),1\right), \qquad \widetilde{S} \coloneqq \left(\left( \begin{matrix} 0 & -1 \\ 1 & 0 \end{matrix}\right) ,\sqrt{\tau}\right),
\end{align*}
and the \emph{Weil representation} $\rho_L$ associated to $L$ is defined on the generators by
\begin{align} \label{eq:Weildef}
\rho_L\left(\widetilde{T}\right)(\ef_h) \coloneqq e^{2\pi iQ(h)} \ef_h, \qquad
\rho_L\left(\widetilde{S}\right)(\ef_h) \coloneqq \frac{e^{-\frac{\pi i}{4}}}{\sqrt{\vt{L'\slash L}}} \sum_{h' \in L'\slash L} e^{-2\pi i(h',h)_Q)} \ef_{h'},
\end{align}
where $\ef_{h}$ for $h \in L'\slash L$ is the standard basis of the group ring $\C[L'\slash L]$. The operator
\begin{align} \label{eq:PeterssonVVdef}
f\big\vert_{\kappa,L}(\gamma,\phi) (\tau) \coloneqq \phi(\tau)^{-2\kappa}\rho_{L}^{-1}(\gamma,\phi)f(\gamma\tau)
\end{align}
captures modularity with respect to $\rho_{L}$ of vector-valued functions. We let $M_{k, L}$ denote the space of vector-valued holomorphic modular forms.

Let $f$ be a smooth vector-valued automorphic form with respect to $\rho_L$ and denote the component functions of $f$ by $f_{h}$. If $N=1$ or $N=p$ is prime, the map
\begin{align} \label{eq:plusspaceprojection}
\sum_{h \pmod*{2N}} f_h(\tau)\ef_h \mapsto \sum_{h \pmod*{2N}} f_h(4N\tau)
\end{align}
projects $f$ to a scalar-valued form of the same weight in Kohnen's plus space, see \cite{eiza}*{Theorem 5.6}, \cite{brufuim}*{Example 2.2}, \cite{bor98}*{Example 2.4}.

The set $\Gamma_0(N) \backslash \Iso(V)$ of isotropic lines $\lf$ (that is $Q(\lf) = 0$) modulo $\Gamma_0(N)$ can be identified with the inequivalent cusps with $i\infty$ corresponding to the isotropic line $\lf_0$ spanned by $u_0 \coloneqq \left(\begin{smallmatrix} 0 & 1 \\ 0 & 0\end{smallmatrix}\right)$. Let $\sigma_{\lf} \in \slz$ be a scaling matrix of the cusp $\lf$, that is $\sigma_{\lf}\lf_0 = \lf$ (unique up to sign). Set $u_\lf = \sigma_{\lf}^{-1}u_0$. Then, there exists some $\beta_{\lf} \in \Q_{>0}$ such that $\beta_{\lf}u_{\lf}$ is a primitive element of $\lf \cap L$ (that is we have $\Q\beta_{\lf}u_{\lf} \cap L = \Z\beta_{\lf}u_{\lf}$). Then, set $\widehat{\varepsilon}_{\lf} \coloneqq \frac{\alpha_{\lf}}{\beta_{\lf}}$, where $\alpha_{\lf}$ is the width of the cusp.

\begin{lemma}[\cite{alfehl}*{Subsection 6.3}] \label{lem:cuspparameters}
Let $N=p$ be a prime.
\begin{enumerate}[label={\rm (\roman*)}]
\item The cusp $i\infty$ has width $\alpha_{\infty} = 1$ and parameters $\beta_{\infty} = \frac{1}{p}$, $\widehat{\varepsilon}_{\infty} = p$.
\item The cusp $0$ has width $\alpha_{0} = p$ and parameters $\beta_{0} = 1$, $\widehat{\varepsilon}_{0} = p$.
\end{enumerate}
\end{lemma}

\subsection{Work of Pei and Wang} \label{subsec:PeiWang}
Let $\mu$ be the M{\"o}bius function, $L(s, \chi)$ be the Dirichlet $L$-function and abbreviate $\chi_{d} \coloneqq \left(\frac{d}{\cdot}\right)$. We recall that
\begin{align}
L_N(s, \chi) &\coloneqq L(\chi,s)\prod_{p\mid N} \left(1-\chi(p)p^{-s}\right) = \prod_{p\nmid N}\frac{1}{1-\chi(p)p^{-s}} = \sum_{\substack{n \geq 1 \\ \gcd(n,N)=1}} \frac{\chi(n)}{n^s}. \label{eq:L_Ndef}
\end{align}
Moreover, let $N$ be odd and square-free, $\ell \mid N$, and
\begin{align} \label{eq:divisorsums}
\sigma_{\ell,N,s}(r) \coloneqq
\sum_{\substack{d \mid r \\ \gcd(d,\ell)=1 \\ \gcd\left(\frac{r}{d},\frac{N}{\ell}\right)=1}} d^s, \qquad \sigma_{N,s}(r) &\coloneqq \sigma_{N,N,s}(r), \qquad \sigma_s(r) \coloneqq \sigma_{1,s}(r),
\end{align}
For $\ell \neq N$, Pei and Wang defined the \emph{generalized Hurwitz class numbers}
\begin{align} \label{eq:HellNdef}
H_{\ell,N}(n)
&\coloneqq \begin{cases}
0 & \text{if } n=0, \\
L_{\ell}\left(0,\chi_{t}\right) \prod\limits_{p\mid \frac{N}{\ell}} \frac{1-\chi_t(p)p^{-1}}{1-p^{-2}} \sum\limits_{\substack{a \mid m \\ \gcd(a,N) = 1}} \mu(a) \chi_t(a) \sigma_{\ell,N,1}\left(\frac{m}{a}\right) & \begin{array}{@{}l} \text{if } -n=tm^2, \\ t \text{ fundamental}, \end{array} \\
0 & \text{else}.
\end{cases}
\end{align}
and
\begin{align} \label{eq:HNNdef}
H_{N,N}(n) 
&\coloneqq \begin{cases}
L_N\left(-1,\mathrm{id}\right) & \text{if } n=0, \\
L_N(0,\chi_t) \sum\limits_{\substack{a \mid m \\ \gcd(a,N) = 1}} \mu(a) \chi_t(a) \sigma_{N,1}\left(\frac{m}{a}\right) & \begin{array}{@{}l} \text{if } -n=tm^2, \\ t \text{ fundamental}, \end{array}  \\
0 & \text{else},
\end{cases}
\end{align}
Here, we modified their original definition of these numbers slightly by including the summation condition $\gcd(a,N) = 1$. This is justified by the remark following \cite{bemo}*{(2.8), (2.9)}. We recall from the introduction that 
$
H_{1,1}(n) = H(n) = \tr_{-n}(1)
$
are the classical Hurwitz class numbers by Dirichlet's class number formula.

Let $E_{\frac{3}{2}}^+(4N)$ denote the Eisenstein plus subspace of $M_{\frac{3}{2}}(4N)$. Pei and Wang \cite{peiwang}*{Theorem 1 (I)} established that 
\begin{align*}
\mathrm{dim}_{\C}\left(E_{\frac{3}{2}}^+(4N)\right) = 2^{\omega(N)}-1, \qquad \omega(N) \coloneqq \sum_{p \mid N} 1,
\end{align*}
and that a basis of $E_{\frac{3}{2}}^+(4N)$ is given by the generating functions $\Hs_{\ell,N}$ of $H_{\ell,N}$ from equation \eqref{eq:CEdef} for $\ell > 1$. In particular, if $N=p$ is an odd prime then $E_{\frac{3}{2}}^+(4p)$ is one-dimensional and generated by $\Hs_{p,p}$. We emphasize once more that both $\Hs_{1,1}$ as well as $\Hs_{1,N}$ are not modular.

\subsection{Scalar-valued Fourier expansions}

Let
\begin{align} \label{eq:Ks0Ndef}
\Ks_{0,N}(m,n;s) \coloneqq \sum_{\substack{c > 0 \\ N \mid c}} \frac{1}{c^{s}} \sum_{\substack{r \pmod*{c} \\ \gcd(c,r)=1}} e^{2\pi i \frac{mr^*+nr}{c}}, \qquad \re(s) > 1,
\end{align}
be the usual integral weight level $N$ \emph{Kloosterman zeta function},  let
\begin{align} \label{eq:Ks0Ntildedef} 
\widetilde{\Ks}_{0,N}(m,n;s) \coloneqq \sum_{\substack{d > 0 \\ \gcd(d,N)=1}} \frac{1}{d^s} \sum_{\substack{r \pmod*{d} \\ \gcd(r,d)=1}} e^{2\pi i \frac{mr^*+nr}{d}}, \qquad \re(s) > 1.
\end{align}
be the \emph{modified Kloosterman zeta function} associated to the same data, and let
\begin{align} \label{eq:Ks4Nplusdef}
\Ks_{k,4N}^+(m,n;s) \coloneqq \sum_{c > 0} \frac{1+\left(\frac{4}{c}\right)}{(4Nc)^{s}} \sum_{\substack{r \pmod*{4Nc} \\ \gcd(4Nc,r)=1}} \left(\frac{4Nc}{r}\right)\varepsilon_r^{2k} e^{2\pi i \frac{mr^*+nr}{4Nc}}, \qquad \re(s) > 1,
\end{align}
for $k \in \Z + \frac{1}{2}$ be the level $4N$ \emph{plus space Kloosterman zeta function}. Let $W_{\mu,\nu}$ be the usual $W$-Whittaker function.
\begin{lemma} \label{lem:Eisensteinexpansions}
Let $\Fc_{k,N}$, $\Fc_{k,4N}^+$ be the non-holomorphic Eisenstein series from equation \eqref{eq:Fcdef}.
\begin{enumerate}[label={\rm (\roman*)}]
\item Let $\re(s) > 1$. The Fourier expansion of $\Fc_{0,N}$ about $i\infty$ is
\begin{multline*}
\hspace*{\leftmargini} \Fc_{0,N}(\tau,s) = v^{s} +  \frac{4^{1-s}\pi \Gamma(2s-1)}{\Gamma\left(s\right)^2} \Ks_{0,N}(0,0;2s) v^{1-s} \\
+ \frac{\pi^{s}}{\Gamma\left(s\right)} \sum_{n \neq 0} \Ks_{0,N}(0,n;2s) \vt{n}^{s-1} W_{0,s-\frac{1}{2}}(4\pi \vt{n} v) e^{2\pi i n u}.
\end{multline*}
\item Let $\re(s) > 1$ and $W_N$ be the Fricke involution. The Fourier expansion of $\Fc_{0,N}$ about $0$ is
\begin{multline*}
\hspace*{\leftmargini} \left(\Fc_{0,N} \big\vert_0 W_N\right) (\tau,s) = \frac{v^s}{N^s}\delta_{N=1} + \frac{4^{1-s} \pi \Gamma(2s-1)}{\Gamma(s)^2 N^s} \widetilde{\Ks}_{0,N}(0,0;2s) v^{1-s} \\
+ \frac{\pi^{s}}{N^s \Gamma(s)} \sum_{n \neq 0} \widetilde{\Ks}_{0,N}(0,n;2s) \vt{n}^{s-1} W_{0,s-\frac{1}{2}}(4\pi\vt{n}v)e^{2\pi i n u}.
\end{multline*}
\item Let $k \in -\frac{1}{2}\N + 1$ and $\re(s) > 1-\frac{k}{2}$. The Fourier expansion of $\Fc_{k,4N}^+$ about $i\infty$ is
\begin{multline*}
\hspace*{\leftmargini} \Fc_{k,4N}^+(\tau,s) = \frac{2}{3}\left(v^{s-\frac{k}{2}} + \frac{4^{1-s}\pi i^{-k}\Gamma(2s-1)}{\Gamma\left(s+\frac{k}{2}\right)\Gamma\left(s-\frac{k}{2}\right)} \Ks_{k,4N}(0,0;2s) v^{1-s-\frac{k}{2}} \right) \\
+ \frac{2}{3}i^{-k}\pi^{s}v^{-\frac{k}{2}} \sum_{\substack{n \neq 0 \\ (-1)^{k-\frac{1}{2}}n \equiv 0,1 \pmod*{4}}}  \frac{\Ks_{k,4N}(0,n;2s) \vt{n}^{s-1}}{\Gamma\left(s+\sgn(n)\frac{k}{2}\right)} W_{\sgn(n)\frac{k}{2},s-\frac{1}{2}}(4\pi \vt{n} v) e^{2\pi i n u}.
\end{multline*}
\end{enumerate}
\end{lemma}

\begin{proof}
Parts (i) and (iii) can be found in \cite{jkk1}*{Theorems 3.2 and 4.4} for example, some details can be found in \cite{bemo}*{Subsection 2.3} as well. To prove part (ii), we use \cite{vass}*{Proposition 6.2.5}. The usual coset representatives for $\Gamma_0(N)_{\infty} \backslash \Gamma_0(N)$ yield
\begin{align*}
\left(\Fc_{0,N} \Big\vert_0 W_N\right) (\tau,s) = \frac{v^{s}}{2N^s} \sum_{\substack{(c, d) \in \Z^2 \setminus \{(0,0)\} \\ \gcd(Nc,d) = 1} } \frac{1}{\vt{c+d\tau}^{2s}} = \frac{v^s}{N^s}\delta_{N=1} + \frac{v^{s}}{2N^s} \sum_{\substack{c \in \Z, \ d \neq 0 \\ \gcd(Nc,d) = 1} } \frac{1}{\vt{c+d\tau}^{2s}}.
\end{align*}
The claim follows by decomposing $c = dj+r$ with $j \in \Z$, $0 \leq r < d$ and using a generalized Lipschitz summation formula (see \cite{maass64}*{pp.\ 207}, \cite{siegel56}*{p.\ 366}). 
\end{proof}

Moreover, we recall
\begin{align} \label{eq:h*def}
h^*(n) = \frac{1}{2\pi} \sum_{r^2 \mid n} \left(2\log\left(\widetilde{\varepsilon}_{\frac{n}{r^2}}\right)\right) h\left(\frac{n}{r^2}\right), \qquad n > 0,
\end{align}
from \cite{ditimrn}*{(2.3)} as well as \cite{alansa}*{(1.5)}. The harmonic part of weight $\frac{1}{2}$ sesquiharmonic Maass forms is given in terms of
\begin{align} \label{eq:alphaspecialdef}
\alpha(y) \coloneqq \sqrt{y} \int_{0}^{\infty} \frac{\log(t+1)}{\sqrt{t}} e^{-\pi y t} \dm t
\end{align}
which is a slight renormalization compared to \cite{dit11annals}*{p.\ 952}.
\begin{lemma}[\cite{dit11annals}*{Theorem 4}, \cite{alansa}*{(1.7)}] \label{lem:EcExpansion}
The function $\Ec$ is a weight $\frac{1}{2}$ sesquiharmonic Maass form on $\Gamma_0(4)$ with Fourier expansion
\begin{multline*}
\Ec(\tau) = \frac{1}{3}v^{\frac{1}{2}} - \frac{1}{4\pi}\log(v) - \frac{1}{\pi}\left(\log(4)-\gamma-\frac{\zeta'(2)}{\zeta(2)}\right) \\
+ \sum_{\substack{d > 0 \\ d \neq \square}} \frac{h^*(d)}{\sqrt{d}} q^d + \sum_{n > 0} \bff\left(n^2\right)q^{n^2} + \frac{1}{4\pi} \sum_{n \neq 0} \alpha\left(4n^2v\right) q^{n^2} + 2v^{\frac{1}{2}} \sum_{d < 0} h^*(d) \Gamma\left(\frac{1}{2}, 4\pi \vt{d}v\right)q^d.
\end{multline*}
\end{lemma}
The coefficients $\bff\left(n^2\right)$ are determined in Theorem \ref{thm:cfsquareindices} (i). We recall the notation
\begin{align} \label{eq:c(n)def}
\cf\left(n\right) \coloneqq \frac{\partial}{\partial s} \left[ \left(s-\frac{3}{4}\right) \Ks_{\frac{1}{2},4N}^+(0,n;2s) \right] \Bigg\vert_{s=\frac{3}{4}}.
\end{align}
from \cite{bemo}*{Section 3}. More detailed expressions for $\Ks_{\frac{1}{2}}^+$ and $\cf(n)$ in the case of $n \neq \square$ can be found there as well.
\begin{lemma}[\cite{bemo}*{Theorem 1.3 (i) and (ii)}] \label{lem:GcExpansion}
Let $N \in \N$ be odd and square-free and $\gamma$ be the Euler--Mascheroni constant. The function $\Gc$ is a weight $\frac{1}{2}$ sesquiharmonic Maass form on $\Gamma_0(4N)$ with Fourier expansion
\begin{multline*}
\Gc(\tau) = \frac{2}{3} v^{\frac{1}{2}} - \frac{\log(16v)}{2\pi} \prod_{p\mid N} \frac{1}{p+1}  + \frac{1}{\pi} \prod_{p\mid N} \frac{1}{p+1} \sum_{m \geq 1 }  \left(\gamma + \log\left(\pi m^2\right) +  \alpha\left(4m^2v\right) \right) q^{m^2} \\
+ \frac{2}{3}(1-i) \pi \sum_{\substack{n \geq 0 \\ n \equiv 0,1 \pmod*{4}}} \cf(n)q^n + \frac{2}{3} (1-i) \pi^{\frac{1}{2}} \sum_{\substack{n < 0 \\ n \equiv 0,1 \pmod*{4}}}  \cf(n) \Gamma\left(\frac{1}{2},4\pi\vt{n}v\right) q^n.
\end{multline*}
\end{lemma}
Part (iii) of \cite{bemo}*{Theorem 1.3} can be found in equation \eqref{eq:GcShadow}, and the coefficients $\cf\left(n^2\right)$ are determined in Theorem \ref{thm:cfsquareindices} (ii) and (iii).

\subsection{Quadratic traces}
Let
\begin{align} \label{eq:Quadraticformssets}
\begin{split}
\mathcal{Q}_{N,n} &\coloneqq \left\{ax^2+bxy+cy^2 \colon a,b,c \in \Z, N \mid a, b^2-4ac=n \right\}, \\
\mathcal{Q}_{N,n,h} &\coloneqq \left\{ax^2+bxy+cy^2 \colon a,b,c \in \Z, N \mid a, b^2-4ac=n, b \equiv h \pmod*{2N} \right\}
\end{split}
\end{align}
and
\begin{align*}
L_{m,h} \coloneqq \left\{X \in L+h \colon Q(X) = m\right\}
\end{align*}
for $m \in \Q^{\times}$ and $h \in L' \slash L$. The group $\Gamma_0(N)$ acts on $L_{m,h}$ with finitely many orbits, and we have the identification
\begin{align} \label{eq:quadraticformsidentification}
\Gamma_0(N) \backslash L_{m,h} \cong \mathcal{Q}_{N,4Nm,h} \slash \Gamma_0(N).
\end{align}
For $m < 0$ and any function $f \colon \Gamma_0(N)\backslash\Hb \to \C$, we define
\begin{align} \label{eq:trimagdef}
\tr_{m,h}(f) \coloneqq \sum_{Q \in \mathcal{Q}_{N,4Nm,h} \slash \Gamma_0(N)} \frac{1}{\vt{\Gamma_0(N)_Q}} f\left(\tau_Q\right), \qquad \tau_{[a,b,c]} \coloneqq \frac{-b}{2a} +i\frac{\vt{\frac{m}{N}}}{2\vt{a}}.
\end{align}
For $m > 0$, we let
\begin{align*}
S_{[a,b,c]} \coloneqq \left\{\tau \in \Hb \colon a\vt{\tau}^2+bu+c=0\right\}
\end{align*}
be the geodesic connecting the two roots of $Q(\tau,1)$. If $m$ is not a square then both roots are quadratic irrationals, and thus $S_Q$ is a half-circle perpendicular to $\R$. The stabilizer $\Gamma_0(N)_Q$ is infinite cyclic and $\Gamma_0(N)\backslash S_Q$ is a closed geodesic in $\Gamma_0(N) \backslash \Hb$. If $m$ is a square, then both roots are rational and one of them is equivalent to $i\infty$. In this case, the stabilizer $\Gamma_0(N)_Q$ is trivial and $\Gamma_0(N)\backslash S_Q$ has infinite length. Letting $f \colon \Gamma_0(N)\backslash\Hb \to \C$ be a continuous function, we define
\begin{align} \label{eq:trrealdef}
\tr_{m,h}(f) \coloneqq \frac{1}{2\pi} \sum_{Q \in \mathcal{Q}_{N,4Nm,h} \slash \Gamma_0(N)} \int_{\Gamma_0(N)\backslash S_Q} f(z) \frac{\dm z}{Q(z,1)},
\end{align}
where the integral has to be regularized whenever $m = \square$. This is described in \cite{brufuim}*{Subsection 3.4}, and we follow their discussion here.

\subsection{Theta functions and theta lifts}
Let $\mathbb{F} \coloneqq \left\{\tau \in \Hb \colon -\frac{1}{2} \leq u < \frac{1}{2}, \vt{\tau} \geq 1\right\}$ be standard fundamental domain for $\slz$, and 
\begin{align*}
\mathbb{F}^a \coloneqq \bigcup_{j=0}^{a-1} \left(\begin{matrix} 1 & j \\ 0 & 1\end{matrix}\right) \mathbb{F}.
\end{align*}
Recall that $\sigma_{\af}$ denotes a scaling matrix of the cusp $\af$ and $\alpha_{\af}$ is the width of $\af$. Then,
\begin{align*}
\mathbb{F}(N) \coloneqq \bigcup_{\af \in \left(\Q \cup \{i\infty\}\right) \slash \Gamma_0(N)} \sigma_{\af} \mathbb{F}^{\alpha_{\af}}
\end{align*}
is a fundamental domain for $\Gamma_0(N)$, which we truncate at height $T > 0$ by letting
\begin{align} \label{eq:funddomtrunc}
\mathbb{F}_T(N) \coloneqq \left\{\tau \in \mathbb{F}(N) \colon v \leq T\right\}, \qquad \mathbb{F}_T^a(N) \coloneqq \left\{\tau \in \mathbb{F}^a \colon v \leq T\right\}.
\end{align}

Let $\varphi_0(X,\tau,z)$ be the standard Gaussian as in \cite{brufuim}*{Subsection 5.1}, and define
\begin{align} \label{eq:thetacompdef}
\theta_{h}(\tau,z,\varphi_0) = \sum_{X \in h + L} \varphi_0(X,\tau,z).
\end{align}
The Siegel theta kernel is given by
\begin{align} \label{eq:Thetakerneldef}
\Theta_L\left(\tau,z,\varphi_0\right) \coloneqq \sum_{h\in L'\slash L} \theta_{h}(\tau,z,\varphi_0) \ef_{h}
\end{align}
It has weight $0$ as a function of $z$ and weight $\frac{1}{2}$ with respect to $\rho_L$ as a function of $\tau$, see \cite{bor98}*{Theorem 4.1}. Consequently, it gives rise to the theta lift (\cite{brufuim}*{Definition 5.6})
\begin{align} \label{eq:ThetaLiftDefn}
I^{\reg}(\tau,f) \coloneqq \CT_{\sfr=0} \left[\sum_{\af \in \{0,i\infty\}} \lim_{T \to \infty} \int_{\mathbb{F}_T(p)^{\alpha_{\af}}} f(\sigma_{\af}z) \Theta_L(\tau,\sigma_{\af}z) y^{-\sfr} \frac{\dm x \dm y}{y^2}\right],
\end{align}
where $f$ is a scalar-valued weight $0$ weak Maass form. The regularization is described in the paragraph surrounding \cite{brufuim}*{(1.16)}, and goes back to Borcherds \cite{bor98} and Harvey--Moore \cite{hamo}. See \cite{brufuim}*{Subsection 5.4} for a detailed discussion.

Following \cite{brufuim}*{(5.4)}, the notation 
\begin{align} \label{eq:ThetaLiftComponents}
I_h^{\reg}(\tau,f) \coloneqq \CT_{\sfr=0} \left[\sum_{\af \in \{0,i\infty\}} \lim_{T \to \infty} \int_{\mathbb{F}_T(p)^{\alpha_{\af}}} f(\sigma_{\af}z) \theta_h(\tau,\sigma_{\af}z,\varphi_0) y^{-\sfr} \frac{\dm x \dm y}{y^2}\right],
\end{align}
refers to the individual components of the vector-valued theta lift $I^{\reg}$ throughout.

Analogously, the Millson theta lift is given by
\begin{align} \label{eq:ThetaLiftMillsonDef}
J^{\reg}(\tau, f) \coloneqq \lim_{T \to \infty} \int_{\mathbb{F}_T(p)} f(z) \Theta_L\left(\tau,z,\varphi_{\text{KM}}\right),
\end{align}
where $f$ is a weight $0$ weak Maass form and $\varphi_{\text{KM}}$ is the Kudla--Millson Schwartz form (see \cite{brufuim}*{Remark 5.3} for example). 

Lastly, for an isotropic line $\lf \in V$, the space $\lf^{\perp} \slash \lf$ is a unary positive quadratic space with quadratic form $Q$, which in turn yields the even lattice
\begin{align} \label{eq:unarylattice}
K_{\lf} \coloneqq \left(L \cap \lf^{\perp}\right) \slash \left(L \cap \lf\right) \subseteq \lf^{\perp} \slash \lf
\end{align}
with dual lattice
\begin{align*}
K_{\lf}' = \left(L' \cap \lf^{\perp}\right) \slash \left(L' \cap \lf\right).
\end{align*}
This construction gives rise to the vector-valued holomorphic theta function
\begin{align*}
\Theta_{K_\lf} \coloneqq \sum_{\lambda \in K_{\lf}'} e^{2\pi i Q(\lambda)\tau} \ef_{\lambda+K_{\lf}} \in M_{\frac{1}{2}, K_{\lf}},
\end{align*}
whose components are $\theta_{K_{\lf},\tilde{h}}(\tau)$ for $\tilde{h} \in K_{\lf}' \slash K_{\lf}$. In turn, these components yield the vector-valued holomorphic modular form
\begin{align} \label{eq:Thetatildedef}
\tilde\Theta_{K_\lf}(\tau) \coloneqq \sum_{\substack{h \in L'\slash L \\ h \perp \lf}} \theta_{K_{\lf},\tilde{h}}(\tau) \ef_h \in M_{\frac{1}{2}, L}
\end{align}
for $\rho_L$, see \cite{brufuim}*{(2.5)}. Following Bruinier, Funke and Imamo\={g}lu, we let $b_{\lf}(m,h)$ be the $(m,h)$-th Fourier coefficient of $\tilde\Theta_{K_\lf}$. 

\subsection{Vector-valued lifts of weight \texorpdfstring{$0$}{0} weak Maass forms}

Bruinier, Funke and Imamo\={g}lu define the special function (incorporating a correction pointed out by Ahlgren, Andersen and Samart \cite{alansa}*{Section 6})
\begin{align} \label{eq:FcSpecialdef}
\Fc(t) \coloneqq \log(t) - \sqrt{\pi} \int_0^t e^{w^2} \erfc(w) \dm w  + \log(2) + \frac{1}{2} \gamma,
\end{align}
where 
\begin{align*}
\erfc(w) = \frac{2}{\sqrt{\pi}} \int_w^{\infty} e^{-t^2} \dm t
\end{align*}
is the complementary error function. This function is Duke, Imamo\={g}lu and T\'{o}th's \cite{dit11annals}*{p.\ 952} special function $\alpha$ (in our normalization from equation \eqref{eq:alphaspecialdef}) essentially.
\begin{lemma}[\cite{alansa}*{Section 6}] \label{lem:specialfunctionrelation}
We have
\begin{align*}
-2\Fc\left(2\sqrt{\pi N v} m\right) = \alpha\left(4Nm^2v\right)
\end{align*}
\end{lemma}

Let $Z(m,h)$ be the Heegner divisor of index $(m,h)$ (see \cite{brufuim}*{p.\ 55}), and 
\begin{align} \label{eq:BFIvol}
\mathrm{vol}(M) \coloneqq -\frac{1}{2\pi} \int_M \dm\mu(z), \qquad \dm\mu(z) \coloneqq \frac{\dm x \dm y}{y^2}.
\end{align}
Given $X \in L'$, let $c(X)$ be the geodesic corresponding to $S_Q$ via the isomorphism \eqref{eq:quadraticformsidentification}.
\begin{lemma}[\cite{brufuim}*{Theorem 4.2}] \label{lem:bfimain2}
Let $h \in L'\slash L$. Define $k_{\lf}$ by $\lf \cap (L+h) = \Z \beta_{\lf} u_{\lf} + k_{\lf}u_{\lf}$ and $0 \leq k_{\lf} < \beta_{\lf}$. Setting $\psi(0) \coloneqq -\gamma$, the function
\begin{align*}
I_h^{\reg}(\tau,1) \coloneqq &-2 \sqrt{v} \mathrm{vol}\left(\Gamma_0(N)\backslash\Hb\right) \delta_{h,0} + \sum_{m<0} \deg Z(m,h)  \frac{\erfc(2\sqrt{\pi \vt{m}v})}{2\sqrt{\vt{m}}} q^m \\
& + \sum_{\substack{m>0\\ m \notin N (\Q^{\times})^2} } \left(\sum_{X \in \Gamma \backslash L_{m,h}} \mathrm{length}(c(X)) \right) q^m + \sum_{m>0} \tr_{Nm^2,h}(1) q^{Nm^2} \\
& - \frac{1}{\sqrt{N}\pi}  \sum_{\lf \in \Gamma \backslash \Iso(V)} \widehat{\varepsilon}_{\lf} \sum_{m>0} b_{\lf}\left(Nm^2,h\right) \Fc\left(2 \sqrt{\pi v N}m\right) q^{Nm^2} \\
& - \frac{1}{2\sqrt{N}\pi} \delta_{h,0}  \sum_{\lf \in \Gamma \backslash \Iso(V)} \widehat{\varepsilon}_{\lf}  \left[\log\left(4 \beta_{\lf}^2 \pi v\right) + \gamma + \psi\left(\frac{k_{\lf}}{\beta_{\lf}}\right) + \psi\left(1-\frac{k_{\lf}}{\beta_{\lf}}\right)\right]
\end{align*}
defines the $h$-component of a weak Maass form for $\rho_L$ of weight $1/2$.
\end{lemma}

Let $L$, $Q$ be as in (resp.\ after) equation \eqref{eq:latticedef} and $\mathrm{Mp}(\Z)_{\infty}$ be the subgroup of $\mathrm{Mp}(\Z)$ generated by $\widetilde{T}$. We define
\begin{align} \label{eq:vvEisDef}
P_{0,0}(\tau,s) \coloneqq \frac{1}{2}\sum_{(\gamma, \phi) \in \mathrm{Mp}(\Z)_\infty \backslash \mathrm{Mp}(\Z)}
\left[ v^{s-\frac{1}{4}} \ef_0\right] \Big\vert_{\frac{1}{2},L}(\gamma,\phi), \qquad \re(s) > \frac{3}{4}.
\end{align}
Note that $P_{0,0}$ takes values in $\C[L'\slash L]$, and is investigated in \cite{alfesthesis}*{Appendix A}. The Fourier expansion of the vector-valued half integral weight Maass--Poincar{\'e} series can be found in \cite{garthwaite}. Noting that $G_0(\tau,s) = \frac{1}{2}\Fc_{0,p}(\tau,s)$, we conclude by citing the following result.
\begin{lemma}{\cite{brufuim}*{Theorem 6.2}} \label{lem:BFIThm6.2}
If $\re(s) > 1$ then
\begin{align*}
I^{\reg}\left(\tau,\Fc_{0,p}(\cdot,s)\right) = p^{\frac{1-s}{2}} \zeta^*(s) P_{0,0}\left(\tau,\frac{s}{2}+\frac{1}{4}\right).
\end{align*}
\end{lemma}

\section{The proof of Theorem \ref{thm:cfsquareindices}} \label{sec:ProofOfThm1.3}

We recall $\sigma_{N,s}$ from equation \eqref{eq:divisorsums} as well as our notation 
\begin{align} \label{eq:TNschidef}
T_{N,s}^{\chi}(n) \coloneqq \sum_{\substack{0 < d \mid n \\ \gcd(d,N)=1}} \mu(d)\chi(d)d^{s-1}\sigma_{N,2s-1}\left(\frac{n}{d}\right)
\end{align}
from \cite{bemo}. We define
\begin{align} \label{eq:t(n)def}
\tf_{4p}(m) &\coloneqq -\frac{1}{2}\frac{\partial}{\partial s} T_{4p,\frac{3}{2}-2s}^{\id}(m) \Big\vert_{s=\frac{3}{4}}
= \sum_{\substack{0 < d \mid m \\ \gcd(d,4p)=1}} \frac{\mu(d)}{d} \sum_{\substack{0 < r \mid \frac{m}{d} \\ \gcd(r,4p)=1}} \frac{\log(d)+2\log(r)}{r}.
\end{align}

\subsection{The case of level \texorpdfstring{$4$}{4}}
We determine the square-indexed Fourier coefficients in the Fourier expansion of $\Ec$. That is, we evaluate
\begin{align*}
\frac{1}{3} \frac{\partial}{\partial s} \left[ \left(s-\frac{3}{4}\right) \Ks_{\frac{1}{2},4}^+(0,n;2s) \right] \Bigg\vert_{s=\frac{3}{4}}
\end{align*}
with $0 \leq n = \square$. We recall that the constant terms of $\Ec$ are given by
\begin{align} \label{eq:DITconstantterms}
\frac{1}{3}\frac{\partial}{\partial s} \left(s-\frac{3}{4}\right)\Fc_{\frac{1}{2},4}^+(\tau,s)\Big\vert_{s=\frac{3}{4}} &= \frac{\sqrt{v}}{3} - \frac{1}{4\pi} \log(v) + \frac{1}{\pi} \left(\gamma - \log(4) - \frac{\zeta'(2)}{\zeta(2)}\right) + \ldots.
\end{align}
To see this, we cite \cite{dit11annals}*{Proposition 2, equation (2.24)}, namely the Fourier expansion
\begin{align*}
\hspace*{\leftmargini} \Fc_{\frac{1}{2},4}^+(\tau,s) = v^{s-\frac{1}{4}} + \frac{2^{2s-\frac{1}{2}}}{(2s-1)\Gamma\left(2s-\frac{1}{2}\right)} \left(\pi^\frac{1}{2} 2^{\frac{5}{2}-6s}\Gamma(2s)\frac{\zeta(4s-2)}{\zeta(4s-1)}\right) v^{\frac{3}{4}-s} + \ldots,
\end{align*}
from which equation \eqref{eq:DITconstantterms} follows by a short computation in accordance with \cite{alansa}*{(1.7)}.

\begin{proof}[Proof of Theorem \ref{thm:cfsquareindices} {\rm (i)}]
We suppose that $1 \leq n = m^2$. According to \cite{andu}*{Proposition 5.7}, we have
\begin{align*}
\Ks_{\frac{1}{2},4}^+\left(0,m^2;2s\right) &= 2^{\frac{3}{2}-4s} e^{2\pi i \frac{1}{8}} m^{2-4s} \frac{\zeta\left(2s-\frac{1}{2}\right)}{\zeta(4s-1)} \sum_{r \mid m} \mu(r) r^{2s-\frac{3}{2}} \sigma_{4s-2}\left(\frac{m}{r}\right) \\
&= \frac{1+i}{2^{4s-1}} \frac{\zeta\left(2s-\frac{1}{2}\right)}{\zeta(4s-1)} \sum_{r \mid m} \mu(r) r^{\frac{1}{2}-2s} \sigma_{2-4s}\left(\frac{m}{r}\right),
\end{align*}
see \cite{bemo}*{Proposition 3.4} (ii) as well.

Including the prefactors in the expansion from Lemma \ref{lem:Eisensteinexpansions} (iii), we inspect
\begin{align*}
\bff\left(m^2\right) = \frac{1}{3} \cdot \frac{2}{3} \cdot \frac{\partial}{\partial s} \left[ \left(s-\frac{3}{4}\right) \frac{\pi^{s+\frac{1}{4}}}{2^{4s-2}} \frac{\zeta\left(2s-\frac{1}{2}\right)}{\zeta(4s-1)} \sum_{r \mid m} \mu(r) r^{\frac{1}{2}-2s} \sigma_{2-4s}\left(\frac{m}{r}\right) \right] \Bigg\vert_{s=\frac{3}{4}}
\end{align*}
in accordance with \cite{dit11annals}*{Proposition 2}. 

\begin{enumerate}
\item We have
\begin{align*}
\lim_{s \to \frac{3}{4}} \left[ \left(s-\frac{3}{4}\right)\zeta\left(2s-\frac{1}{2}\right) \right] = \frac{1}{2}, \qquad
\frac{\partial}{\partial s} \left[ \left(s-\frac{3}{4}\right)\zeta\left(2s-\frac{1}{2}\right) \right] \Bigg\vert_{s=\frac{3}{4}} = \gamma.
\end{align*}
\item We compute that
\begin{align*}
\frac{\partial}{\partial s} \left[\frac{\pi^{s+\frac{1}{4}}}{2^{4s-2}} \frac{1}{\zeta(4s-1)}\right] \Bigg\vert_{s=\frac{3}{4}} = - \frac{12}{\pi} \left(\frac{\zeta'(2)}{\zeta(2)} + \log(2) - \frac{1}{4}\log(\pi)\right).
\end{align*}
\item By \cite{bemo}*{Lemma 3.6}, we have
\begin{align*}
\lim_{s \to \frac{3}{4}} \left[\sum_{r \mid m} \mu(r) r^{\frac{1}{2}-2s} \sigma_{2-4s}\left(\frac{m}{r}\right)\right] = T_{1,0}^{\id}(m) = 1.
\end{align*}
\item We compute that
\begin{align*}
\frac{\partial}{\partial s} \left[\sum_{r \mid m} \mu(r) r^{\frac{1}{2}-2s} \sigma_{2-4s}\left(\frac{m}{r}\right) \right] \Bigg\vert_{s=\frac{3}{4}} = -2\tf_{1}(m).
\end{align*}
\end{enumerate}
Combining expressions yields the claim.
\end{proof}

\subsection{The case of level \texorpdfstring{$4p$}{4p}}

\begin{proof}[Proof of Theorem \ref{thm:cfsquareindices} {\rm (ii)} and {\rm (iii)}]
Following \cite{bemo}*{Subsection 3.4}, we let
\begin{align*}
\ff\left(m^2,s\right) \coloneqq \left(\sum_{j\geq2} \frac{a\left(2^j,m^2\right)}{2^{js}} + \frac{1+i}{2^{2s}}\right) \sum_{j\geq1} \frac{a\left(p^j,m^2\right)}{p^{js}} \begin{cases}
1 & \text{if } m = 0, \\
T_{4N,\frac{3}{2}-s}^{\id}(m) & \text{if } m \geq 1,
\end{cases}
\end{align*}
where we have used that $N=p$ is an odd prime. We prove both parts separately.

\begin{enumerate}[label={\rm (\roman*)}]
\setcounter{enumi}{1}
\item By virtue of \cite{bemo}*{(3.10)}, we have
\begin{multline*}
\hspace*{\leftmargini} \cf(0) = \left(\frac{\partial}{\partial s} \left(s-\frac{3}{4}\right) \frac{L_{4p}(4s-2,\id)}{L_{4p}(4s-1, \id)}\right) \Bigg\vert_{s=\frac{3}{4}} \ff\left(0,\frac{3}{2}\right) \\
+ \left(\lim_{s \to \frac{3}{4}^+} \left(s-\frac{3}{4}\right) \frac{L_{4p}(4s-2,\id)}{L_{4p}(4s-1, \id)}\right) \frac{\partial}{\partial s} \ff(0,2s) \Big\vert_{s=\frac{3}{4}}.
\end{multline*}
First, we observe that
\begin{align*}
a\left(2^j,0\right) &= \sum_{r=1}^{2^j} \left(\frac{2^j}{r}\right)\varepsilon_{r} = \begin{cases}
1 & \text{if } j = 1, \\
(1+i) 2^{j-2} & \text{if } j \geq 2 \text{ is even}, \\
0 & \text{if } j \geq 3 \text{ is odd},
\end{cases}
\\
a\left(p^j,0\right) &= \varepsilon_{p^j}^{-1} \sum\limits_{r=1}^{p^j} \left(\frac{r}{p^j}\right) = \begin{cases}
\varphi\left(p^j\right) & \text{if } j \geq 2 \text{ is even}, \\
0 & \text{if } j \geq 1 \text{ is odd},
\end{cases}
\end{align*}
where $\varphi$ denotes Euler's totient function. Now, we analyze each factor separately.
\begin{enumerate}
\item By \cite{bemo}*{Proposition 3.5 (ii)}, we have
\begin{align*}
\ff\left(0,\frac{3}{2}\right) = \frac{1}{p}\left(\frac{1+i}{4} + \frac{1+i}{8}\right) = \frac{3}{8p}(1+i).
\end{align*}
\item We compute that
\begin{multline*}
\hspace*{\leftmargini+\leftmarginii}
 \frac{\partial}{\partial s} \ff(0,2s) = \frac{\partial}{\partial s} \left(\sum_{j\geq2} \frac{a\left(2^j,0\right)}{2^{2js}} + \frac{1+i}{2^{4s}}\right) \sum_{j\geq1} \frac{a\left(p^j,0\right)}{p^{2js}} \\
= -2\log(2)\left(\sum_{j\geq2} \frac{a\left(2^j,0\right) j}{2^{2js}} + 2\frac{1+i}{2^{4s}}\right) \sum_{j\geq1} \frac{a\left(p^j,0\right)}{p^{2js}} \\
- 2 \log(p) \left(\sum_{j\geq2} \frac{a\left(2^j,0\right)}{2^{2js}} + \frac{1+i}{2^{4s}}\right) \sum_{j\geq1} \frac{a\left(p^j,0\right) j}{p^{2js}}.
\end{multline*}
Evaluating at $s=\frac{3}{4}$ and using \cite{bemo}*{Proposition 3.5} yields
\begin{multline*}
\hspace*{\leftmargini+\leftmarginii} 
\frac{\partial}{\partial s} \ff(0,2s) \Big\vert_{s=\frac{3}{4}} = -2\log(2) \left(\sum_{j\geq2} \frac{a\left(2^j,0\right) j}{2^{\frac{3}{2}j}} + \frac{1+i}{4}\right) \frac{1}{p} \\
- 2 \log(p) \left(\frac{1+i}{4} + \frac{1+i}{8}\right) \sum_{j\geq1} \frac{a\left(p^j,0\right) j}{p^{\frac{3}{2}j}}.
\end{multline*}
Moreover, we have
\begin{align*}
\sum_{j\geq2} \frac{a\left(2^j,0\right) j}{2^{\frac{3}{2}j}} = (1+i), \qquad
\sum_{j\geq1} \frac{a\left(p^j,0\right) j}{p^{\frac{3}{2}j}} = \frac{2}{p-1},
\end{align*}
and we infer that
\begin{align*}
\frac{\partial}{\partial s} \ff(0,2s) \Big\vert_{s=\frac{3}{4}} = -\frac{5(1+i)\log(2)}{2p} - \frac{3}{2}(1+i)\frac{\log(p)}{p-1}.
\end{align*}
\item We compute that
\begin{align*}
%\hspace*{\leftmargini+\leftmarginii}
 \lim_{s \to \frac{3}{4}^+} \left(s-\frac{3}{4}\right) \frac{L_{4p}(4s-2,\id)}{L_{4p}(4s-1, \id)} = \frac{p}{(p+1) \pi^2}.
\end{align*}
\item Since $\frac{\partial}{\partial s} \left[ \left(s-1\right) \zeta(s) \right] \Big\vert_{s=1} = \gamma$, we obtain
\begin{multline*}
\hspace*{\leftmargini+\leftmarginii} 
\frac{\partial}{\partial s} \left(s-\frac{3}{4}\right) \frac{L_{4p}(4s-2,\id)}{L_{4p}(4s-1, \id)} \Big\vert_{s=\frac{3}{4}} \\
= \gamma \frac{4p}{\pi^2(p+1)} + \frac{1}{4} \frac{16p\left(\left(p^2-1\right)\left(\pi^2\log(4)-18\zeta'(2)\right)+3\pi^2p\log(p)\right)}{3\pi^4(p-1)(p+1)^2}.
\end{multline*}
\end{enumerate}
Combining expressions, we obtain
\begin{multline*}
\hspace*{\leftmargini} 
\cf(0) = \frac{3}{8p}(1+i) \left[\gamma \frac{4p}{\pi^2(p+1)} + \frac{1}{4} \frac{16p\left(\left(p^2-1\right)\left(\pi^2\log(4)-18\zeta'(2)\right)+3\pi^2p\log(p)\right)}{3\pi^4(p-1)(p+1)^2}\right] \\
+ \frac{p}{(p+1) \pi^2}\left[-\frac{5(1+i)\log(2)}{2p} - \frac{3}{2}(1+i)\frac{\log(p)}{p-1}\right],
\end{multline*}
which implies the claim after some simplifications.

\item Following \cite{bemo}*{(3.11)}, we have
\begin{multline*}
\hspace*{\leftmargini} 
\cf\left(m^2\right) = \left(\frac{\partial}{\partial s} \left(s-\frac{3}{4}\right) \frac{L_{4p}\left(2s-\frac{1}{2},\id\right)}{L_{4p}(4s-1, \id)}\right) \Bigg\vert_{s=\frac{3}{4}} \ff\left(m^2,\frac{3}{2}\right) \\
+ \left(\lim_{s \to \frac{3}{4}^+} \left(s-\frac{3}{4}\right) \frac{L_{4p}\left(2s-\frac{1}{2},\id\right)}{L_{4p}(4s-1, \id)}\right) \frac{\partial}{\partial s} \ff\left(m^2,2s\right) \Big\vert_{s=\frac{3}{4}},
\end{multline*}
and analyze each term separately.
\begin{enumerate}
\item By \cite{bemo}*{Proposition 3.5 (ii) and Lemma 3.6}, we have
\begin{align*}
\ff\left(m^2,\frac{3}{2}\right) = \frac{1}{p}\left(\frac{1+i}{4} + \frac{1+i}{8}\right) = \frac{3}{8p}(1+i).
\end{align*}
Since $T_{4p,0}^{\id}(m) = 1$, the result is the same as in part (i) (a) above.
\item By \cite{bemo}*{Proposition 3.5 (ii) and Lemma 3.6}, we obtain
\begin{multline*}
\hspace*{\leftmargini+\leftmarginii} 
\frac{\partial}{\partial s} \ff\left(m^2,2s\right) \Bigg\vert_{s=\frac{3}{4}} = \frac{\partial}{\partial s} \left[ \left(\sum_{j\geq2} \frac{a\left(2^j,m^2\right)}{2^{2js}} + \frac{1+i}{2^{4s}}\right) \sum_{j\geq1} \frac{a\left(p^j,m^2\right)}{p^{2js}} \right] \Bigg\vert_{s=\frac{3}{4}} \\
+  \left(\frac{1+i}{4} + \frac{1+i}{8}\right) \frac{1}{p}
\left(\frac{\partial}{\partial s} T_{4p,\frac{3}{2}-2s}^{\id}(m) \Big\vert_{s=\frac{3}{4}}\right).
\end{multline*}
By \cite{bemo}*{Proposition 3.5 (ii)}, the first summand simplifies to
\begin{multline*}
\hspace*{\leftmargini+\leftmarginii} 
\frac{\partial}{\partial s} \left[ \left(\sum_{j\geq2} \frac{a\left(2^j,m^2\right)}{2^{2js}} + \frac{1+i}{2^{4s}}\right) \sum_{j\geq1} \frac{a\left(p^j,m^2\right)}{p^{2js}} \right] \Bigg\vert_{s=\frac{3}{4}} \\
= \frac{1}{p} \frac{\partial}{\partial s} \left[ \left(\sum_{j\geq2} \frac{a\left(2^j,m^2\right)}{2^{2js}} + \frac{1+i}{2^{4s}}\right) \right] \Bigg\vert_{s=\frac{3}{4}} 
+ \left(\frac{1+i}{4} + \frac{1+i}{8}\right) \frac{\partial}{\partial s} \sum_{j\geq1} \frac{a\left(p^j,m^2\right)}{p^{2js}} \Big\vert_{s=\frac{3}{4}}.
\end{multline*}
Note that $\nu_p(m^2)$ is even for all primes $p$ and both $x^2 \equiv 3 \pmod*{4}$, $x^2 \equiv 5 \pmod*{8}$ have no integer solutions. Thus, \cite{bemo}*{(3.5), (3.6)} states that
\begin{multline*}
\hspace*{\leftmargini+\leftmarginii} 
\sum_{j\geq2} \frac{a\left(2^{j},m^2\right) }{2^{2js}} = \frac{1+i}{2^{2s+\frac{1}{2}}} \frac{2^{-(2\nu_2(m)+2)\left(2s-\frac{1}{2}\right)}}{2^{4s-1}-2} \\
\times \left[2^{\nu_2(m)+1}\left(2^{2s-\frac{1}{2}}-2\right)\left(2^{2s-\frac{1}{2}}+1\right)+2^{(2\nu_2(m)+3)\left(2s-\frac{1}{2}\right)} \right]
\end{multline*}
and
\begin{multline*}
\hspace*{\leftmargini+\leftmarginii} 
\sum_{j\geq1} \frac{a\left(p^{j},m^2\right) }{p^{2js}} = -1 + \frac{1}{p^{4s-1} - p} \left[
-p^{\nu_p(m)-2\nu_p(m)\left(2s-\frac{1}{2}\right)+(4s-1)} \right. \\ \left.
+p^{\nu_p(m)-2\nu_p(m)\left(2s-\frac{1}{2}\right)}\left(p^{4s-1} - p^{\frac{3}{2}-2s} + p^{2s-\frac{1}{2}} - p + 1\right) + p^{4s-1} - 1 \right]
\end{multline*}
for $p > 2$. We have
\begin{multline*}
\hspace*{\leftmargini+\leftmarginii}
\frac{\partial}{\partial s} \left[ \left(\sum_{j\geq2} \frac{a\left(2^j,m^2\right)}{2^{2js}} + \frac{1+i}{2^{4s}}\right) \right] \Bigg\vert_{s=\frac{3}{4}} 
= \frac{\partial}{\partial s} \Bigg\{\frac{1+i}{2^{2s+\frac{1}{2}}} \cdot \frac{2^{-(2\nu_2(m)+2)\left(2s-\frac{1}{2}\right)}}{2^{4s-1}-2} \\
\times \left[2^{\nu_2(m)+1}\left(2^{2s-\frac{1}{2}}-2\right)\left(2^{2s-\frac{1}{2}}+1\right) 
+2^{(2\nu_2(m)+3)\left(2s-\frac{1}{2}\right)} \right]
+ \frac{1+i}{2^{4s}}
\Bigg\} \Bigg\vert_{s=\frac{3}{4}} \\
= \frac{1+i}{4}\left(3\cdot2^{-\nu_2(m)}-10\right) \log(2),
\end{multline*}
and
\begin{multline*}
\hspace*{\leftmargini+\leftmarginii} 
\frac{\partial}{\partial s} \sum_{j\geq1} \frac{a\left(p^j,m^2\right)}{p^{2js}} \Big\vert_{s=\frac{3}{4}} 
= \frac{\partial}{\partial s} \left[-1 + \frac{1}{p^{4s-1} - p} \left(-p^{\nu_p(m)-2\nu_p(m)\left(2s-\frac{1}{2}\right)+(4s-1)} 
\right. \right. \\ \left. \left.
\hspace*{\leftmargini+\leftmarginii}  + p^{\nu_p(m)-2\nu_p(m)\left(2s-\frac{1}{2}\right)}\left(p^{4s-1} - p^{\frac{3}{2}-2s} + p^{2s-\frac{1}{2}} - p + 1\right) + p^{4s-1} - 1\right) \right] \Big\vert_{s=\frac{3}{4}} \\
= \frac{2(p+1)p^{-\nu_p(m)-1}\log(p)}{p-1} - \frac{4\log(p)}{p-1}.
\end{multline*}
Moreover, we have
\begin{align*}
\hspace*{\leftmargini+\leftmarginii}
 \frac{\partial}{\partial s} T_{4p,\frac{3}{2}-2s}^{\id}(m)\Big\vert_{s=\frac{3}{4}} 
= \frac{\partial}{\partial s} \sum_{\substack{0 < d \mid m \\ \gcd(d,4p)=1}} \mu(d) d^{\frac{1}{2}-2s} \sum_{\substack{0 < r \mid \frac{m}{d} \\ \gcd(r,4p)=1}} r^{2-4s} \Big\vert_{s=\frac{3}{4}} 
= -2\tf_{4p}(m).
\end{align*}

Combining, we infer
\begin{multline*}
\hspace*{\leftmargini+\leftmarginii}
 \frac{\partial}{\partial s} \ff\left(m^2,2s\right) \Big\vert_{s=\frac{3}{4}} = 
\frac{1+i}{4p}\left(3\cdot2^{-\nu_2(m)}-10\right) \log(2) \\
+ \left(\frac{1+i}{4} + \frac{1+i}{8}\right) \left(\frac{2(p+1)p^{-\nu_p(m)-1}\log(p)}{p-1} - \frac{4\log(p)}{p-1}\right)
 + \frac{3}{8p}(1+i) \left(-2 \tf_{4p}(m)\right).
\end{multline*}

\item We compute that
\begin{align*}
%\hspace*{\leftmargini}
 \lim_{s \to \frac{3}{4}^+} \left(s-\frac{3}{4}\right) \frac{L_{4p}\left(2s-\frac{1}{2},\id\right)}{L_{4p}(4s-1, \id)} = \frac{2p}{\pi^2(p+1)}.
\end{align*}

\item We compute that
\begin{multline*}
\hspace*{\leftmargini+\leftmarginii} 
\frac{\partial}{\partial s} \left(s-\frac{3}{4}\right) \frac{L_{4p}\left(2s-\frac{1}{2},\id\right)}{L_{4p}(4s-1, \id)} \Bigg\vert_{s=\frac{3}{4}} \\
= \frac{6\pi^2\gamma-72\zeta'(2)}{\pi^4} \frac{2p}{3(p+1)} + \frac{3}{\pi^2} \left(\frac{4p\log(p)}{3(p+1)^2} + \frac{4p\log(2)}{9(p+1)}\right).
\end{multline*}
\end{enumerate}
Combining expressions, we obtain
\begin{multline*}
\hspace*{\leftmargini} 
\cf\left(m^2\right) = \frac{3}{8p}(1+i)\left[\frac{6\pi^2\gamma-72\zeta'(2)}{\pi^4} \frac{2p}{3(p+1)} + \frac{3}{\pi^2} \left(\frac{4p\log(p)}{3(p+1)^2} + \frac{4p\log(2)}{9(p+1)}\right)\right] \\
+ \frac{2p}{\pi^2(p+1)}\left[\left(\frac{1+i}{4} + \frac{1+i}{8}\right) \left(\frac{2(p+1)p^{-\nu_p(m)-1}\log(p)}{p-1} - \frac{4\log(p)}{p-1}\right) \right. \\ \left.
+ \frac{1+i}{4p}\left(3\cdot2^{-\nu_2(m)}-10\right) \log(2) + \frac{3}{8p}(1+i) \tf_{4p}(m)\right],
\end{multline*}
and simplifying yields the claim. \qedhere
\end{enumerate}
\end{proof}

\section{The proof of Theorem \ref{thm:GcAsThetaLift}} \label{sec:ProofOfThm1.1}

\subsection{Proof of part (i)}
Part (i) of Theorem \ref{thm:GcAsThetaLift} follows by the following result essentially.
\begin{prop} \label{prop:sesquiharmonicequality}
Let $p$ be prime. Let $f_1$, $f_2$ be two scalar-valued sesquiharmonic Maass forms of weight $\frac{1}{2}$ and level $4p$ in the plus-space. Suppose that $f_1(\tau) - f_2(\tau)$ is of linear exponential decay as $v \to \infty$ and of at most moderate growth towards all cusps inequivalent to $i\infty$. Then, we have $f_1 = f_2$.
\end{prop}

\begin{proof}[Proof of Proposition \ref{prop:sesquiharmonicequality}]
Let $F(\tau) \coloneqq f_1(\tau) - f_2(\tau)$. By virtue of our assumptions on $f_1$ and $f_2$ (in particular the growth towards the cusps inequivalent to $i\infty$), we have
\begin{align*}
\Delta_{\frac{1}{2}}F \in M_{\frac{1}{2}}(4p).
\end{align*}
The Serre--Stark basis theorem \cite{sest} (see \cite{thebook}*{Theorem 2.8} too) implies that $\Delta_{\frac{1}{2}}F = \alpha \cdot \theta$ for some constant $\alpha\in\C$. Since $\Delta_{\frac{1}{2}}F$ has no constant term at $i\infty$ while $\theta$ equals $1$ at $i\infty$, we must have $\alpha = 0$. Hence, $F$ is a harmonic Maass form of weight $\frac{1}{2}$ and level $4p$ satisfying the plus space condition. We deduce
\begin{align*}
\xi_{\frac{1}{2}}F \in M_{\frac{3}{2}}^+(4p) = S_{\frac{3}{2}}^+(4p) \oplus E_{\frac{3}{2}}^+(4p),
\end{align*}
since $\xi_{\frac{1}{2}}F$ is holomorphic at all other cusps as well. As $p$ is prime, the Eisenstein space $E_{\frac{3}{2}}^+(4p)$ is one-dimensional and spanned by a single Eisenstein series according to the work of Pei and Wang \cite{peiwang}. In particular, this Eisenstein series does not vanish at $i\infty$. However, $F$ and therefore $\xi_{\frac{1}{2}}F$ vanish at $i\infty$, for which reason there is no contribution from $E_{\frac{3}{2}}^+(4p)$ to $\xi_{\frac{1}{2}}F$. This establishes that
\begin{align*}
\xi_{\frac{1}{2}}F \in S_{\frac{3}{2}}^+(4p).
\end{align*}
Let $c_{f,\af}(n)$ denote the $n$-th Fourier coefficient of $f$ about the cusp $\af$. We employ the Bruinier--Funke pairing and \cite{brufu02}*{Proposition 3.5} getting
\begin{align*}
\left\langle \xi_{\frac{1}{2}}F, \xi_{\frac{1}{2}}F \right\rangle = \left\{\xi_{\frac{1}{2}}F,F\right\} = \sum_{\af \in \{0,i\infty\}} \sum_{n \leq 0} c_{\xi_{\frac{1}{2}}F,\af}(-n) c_{F,\af}^+(n),
\end{align*}
where $\langle \cdot, \cdot \rangle$ is the Petersson inner product. By assumption, we have $c_{F,i\infty}^+(n) = 0$ for every $n \leq 0$, and hence
\begin{align*}
\langle \xi_{\frac{1}{2}}F, \xi_{\frac{1}{2}}F \rangle = \sum_{n \leq 0} c_{\xi_{\frac{1}{2}}F,0}(-n) c_{F,0}^+(n).
\end{align*}
Since $F$ has moderate growth at all cusps and $\xi_{\frac{1}{2}}F$ is cuspidal, we obtain
\begin{align*}
\langle \xi_{\frac{1}{2}}F, \xi_{\frac{1}{2}}F \rangle = 0,
\end{align*}
and this implies that $F\in M_{\frac{1}{2}}(4p)$ by non-degeneracy of the Petersson inner product. The claim now follows by the Serre--Stark basis theorem as outlined above.
\end{proof}

Theorem \ref{thm:GcAsThetaLift} (i) is a direct consequence of Proposition \ref{prop:sesquiharmonicequality}.
\begin{proof}[Proof of Theorem \ref{thm:GcAsThetaLift} {\rm (i)}]
The constant term of a sesquiharmonic Maass form satisfying the given conditions is a linear combination of the functions $1$, $\log(v)$, and $v^{\frac{1}{2}}$, which with the previous result implies that the space is at most three-dimensional. Since the constant terms of $\mathcal{E}, \mathcal{G}$, and $\theta$ are linearly independent, the result follows from Proposition \ref{prop:sesquiharmonicequality}. 
\end{proof}

\subsection{Spectral deformations}
The proof of Theorem \ref{thm:GcAsThetaLift} (ii) requires spectral deformations of $1$. This is motivated by the discussion in \cite{brufuim}*{Subsubsection 5.5.1, Subsection 6.4}.
\begin{lemma} \label{lem:spectraldeformations}
Let $p$ be an odd prime. The family of functions
\begin{align*}
f_s(\tau) \coloneqq \frac{\pi(p+1)}{6}\frac{1}{\zeta^*(2s-1)} \Fc_{0,p}(\tau,s)
\end{align*}
defines a spectral deformation of the constant function $f=1$. That is, the functions $f_s$ are 
\begin{itemize}
\item holomorphic (with respect to $s$) in an open neighborhood of $s=1$,
\item weight $0$ and level $p$ weak Maass forms with eigenvalue $s(1-s)$ as a function of $\tau$,
\end{itemize}
and satisfy $\lim_{s \to 1}f_s(\tau) = f = 1$.
\end{lemma}

\begin{proof}
Since both $\Fc_{0,p}(\tau,s)$ and $\zeta^*(2s-1)$ have a simple pole at $s=1$, the function $f_s(\tau)$ has an isolated and removable singularity at $s=1$. The properties as a function of $\tau$ are clear by construction. The final assertion follows using \cite{goldstein}*{Theorems 3-1 and 3-3} to show
\begin{align*}
\lim_{s \to 1} (s-1)\Fc_{0,p}(\tau,s) = \frac{3}{\pi(p+1)},
\end{align*}
since $\Gamma_0(p)$ has index $p+1$ in $\slz$.
\end{proof}

Lemma \ref{lem:spectraldeformations} enables us to prove a non-explicit analog of \cite{brufuim}*{Corollary 6.3}. To this end, we define $B_{\af}(s)$ by
\begin{align} \label{eq:Bafdef}
\left(f_s\big\vert_0\sigma_{\af}\right)(\tau) \eqqcolon B_{\af}(s) v^{1-s} + \ldots,
\end{align}
following \cite{brufuim}*{Proposition 5.15}, and we recall $P_{0,0}$ from equation \eqref{eq:vvEisDef} as well as $\tilde\Theta_{K_{\af}}$ from equation \eqref{eq:Thetatildedef}.
\begin{cor} \label{cor:bfiCor6.3}
We have
\begin{multline*}
I^{\reg}(\tau,1) = \frac{\pi (p+1)}{6} \CT_{s=1} \left[\frac{p^{\frac{1-s}{2}}\zeta^*(s)}{\zeta^*(2s-1)}P_{0,0}\left(\tau,\frac{s}{2}+\frac{1}{4}\right)\right] \\
- \sqrt{p}\left(\frac{\partial}{\partial s}B_{\infty}(s)\Big\vert_{s=1} \cdot \tilde\Theta_{K_{\infty}}(\tau) + \frac{\partial}{\partial s}B_{0}(s)\Big\vert_{s=1} \cdot \tilde\Theta_{K_0}(\tau)\right).
\end{multline*}
\end{cor}

\begin{proof}
According to Lemma \ref{lem:cuspparameters}, we have $\widehat{\varepsilon}_{0} = \widehat{\varepsilon}_{\infty} = p$. We apply \cite{brufuim}*{Proposition 5.15} getting
\begin{multline*}
\CT_{s=1} \left[\frac{1}{s(1-s)} \int_{\mathbb{F}(p)} f_s(z) \Theta_L\left(\tau,z, \Delta_{\frac{1}{2},z}\varphi_0\right)\dm\mu(z)\right] \\
= \CT_{s=1} \left[\int_{\mathbb{F}(p)} f_s(z) \Theta_L\left(\tau,z,\varphi_0\right)\dm\mu(z)\right] = \frac{\pi(p+1)}{6} \CT_{s=1} \left[ \frac{I^{\reg}\left(\tau, \Fc_{0,p}(\cdot,s)\right)}{\zeta^*(2s-1)}  \right],
\end{multline*}
and the claim follows according to Lemma \ref{lem:BFIThm6.2}.
\end{proof}

We will make Corollary \ref{cor:bfiCor6.3} explicit in Section \ref{sec:ProofOfThm1.5}. 

\subsection{Proof of part (ii)}

We require another result to prove Theorem \ref{thm:GcAsThetaLift} (ii). To this end, we write
\begin{align*}
P_{0,0}(\tau,s) = \sum_{h \pmod*{2p}} P_{0,0}^{(h)} (\tau,s) \ef_h
\end{align*}
for the individual components of the vector-valued Eisenstein series $P_{0,0}$ from equation \eqref{eq:vvEisDef}. The final sentence of \cite{brufuim}*{Subsection 6.2} states that the level $4$ projection (see equation \eqref{eq:plusspaceprojection}) of $P_{0,0}$ to the plus space yields $\Fc_{\frac{1}{2},4}^+$. Analogously, we show that the level $4p$ projection yields $\Fc_{\frac{1}{2},4p}^+$ for odd primes $p$.
\begin{lemma} \label{lem:Eisensteinprojection}
Suppose that $p$ is prime and that $\re(s) > \frac{3}{4}$. Then, we have
\begin{align*}
\frac{1}{3}\sum_{h \pmod*{2p}} P_{0,0}^{(h)}(4p\tau,s) = (4p)^{s-\frac{1}{4}}\Fc_{\frac{1}{2},4p}^+(\tau,s).
\end{align*}
\end{lemma}

\begin{proof}
Let $\widetilde{\Gamma}$ be the metaplectic double cover of a congruence group $\Gamma < \slz$ (defined analogously to $\mathrm{Mp}(\Z)$ in equation \eqref{eq:Mp2Zdef}). According to Borcherds \cite{bor00} and Scheithauer \cite{scheit}*{Proposition 6.1}, a scalar-valued modular form $f$ of weight $\kappa \in \frac{1}{2}\Z$ and level $4N$ gives rise to a vector-valued modular form
\begin{align*}
\sum_{\left(\gamma, \phi\right) \in \widetilde{\Gamma_0(4N)} \backslash \mathrm{Mp}_2(\Z)} \left(f \vert_{\kappa} \gamma\right)(\tau) \rho_{L}\left(\gamma^{-1}\right) \ef_{0}
\end{align*}
of weight $\kappa$ for $\rho_{L}$. Here, we used the Petersson slash operator for scalar-valued modular forms. Inserting $f(\tau) = (4p)^{s-\frac{1}{4}}\Fc_{\frac{1}{2},4p}^+(\tau,s)$ yields $ P_{0,0}(4p\tau,s)$, where the multiple of $\frac{1}{3}$ can be deduced easily by comparing the constant terms of both Fourier expansions about $i\infty$, see Lemma \ref{lem:Eisensteinexpansions} (iii) and \cite{alfesthesis}*{Proposition A.0.4}. This is equivalent to the claim, because $p$ is prime, see \cite{alfesthesis}*{Subsection 7.3} (last line on page $114$).
\end{proof}

Alternatively, one can match the Fourier expansion about $i\infty$ on each side of the identity directly. To this end, we refer to work by Shintani \cite{shin}, Bruinier and Kuss \cite{bruku}, as well as our previous work \cite{bemo}*{Subsection 3.2} to match the Kloosterman zeta functions in front of $v^{1-s}$.

Now, we are in position to prove Theorem \ref{thm:GcAsThetaLift} (ii).
\begin{proof}[Proof of Theorem \ref{thm:GcAsThetaLift} {\rm (ii)}]
Using equation \eqref{eq:plusspaceprojection}, we let
\begin{align*}
F(\tau) \coloneqq \frac{2}{p-1}\Ec(\tau) + \Gc(\tau) - \Cf(p) \theta(\tau) - \frac{\sqrt{p}}{p^2-1}\sum_{h \pmod*{2p}} I_h^{\reg}(4p\tau,1).
\end{align*}
The functions $\Ec$, $\Gc$, and $\theta$ are weight $\frac{1}{2}$ sesquiharmonic Maass forms of level $4p$ and satisfy the plus space condition (see \cite{dit11annals}*{Theorem 4} and \cite{bemo}*{Theorem 1.3}). The theta lift of $1$ is a vector-valued weight $\frac{1}{2}$ sesquiharmonic Maass form of level $4p$ according to \cite{brufuim}*{Theorem 4.1}, and we project it to a scalar-valued form in the plus space as $p$ is prime. Thus, it suffices to check that $F$ is of exponential decay as $v \to \infty$ and that it has at most moderate growth towards all other inequivalent cusps. Then, the claim follows by virtue of Proposition \ref{prop:sesquiharmonicequality}.

Let $\af$ be a cusp not equivalent to $i\infty$. First, we establish the growth condition towards $\af$. We recall that both $\Ec$ and $\Gc$ arise as the constant term in the Laurent expansions of the Eisenstein series $\Fc_{\frac{1}{2},4}^+$ and $\Fc_{\frac{1}{2},4p}^+$ about $s=\frac{3}{4}$ respectively. According to \cite{kubota}*{Theorem 2.1.2}, both Eisenstein series are bounded at $\af$ whenever $\re(s) > \frac{3}{4}$. Here, we note that the proof of this result generalizes from weight $0$ to weight $\frac{1}{2}$ straightforwardly. Thus, both $\Ec$ and $\Gc$ are bounded towards $\af$, and clearly $\theta$ is of at most moderate growth towards $\af$ as well. To show that the projection of the theta lift to the plus space is of at most moderate growth towards $\af$, we project the identity from Corollary \ref{cor:bfiCor6.3} to the plus space. Combining with the Serre--Stark basis theorem \cite{sest} (see \cite{thebook}*{Theorem 2.8} too), this yields
\begin{align*}
\sum_{h \pmod*{2p}} I_h^{\reg}(4p\tau,1) = \frac{\pi (p+1)}{6} \CT_{s=1} \left[\frac{p^{\frac{1-s}{2}}\zeta^*(s)}{\zeta^*(2s-1)}\sum_{h \pmod*{2p}}P_{0,0}\left(4p\tau,\frac{s}{2}+\frac{1}{4}\right)\right] + \lambda \theta(\tau)
\end{align*}
for some $\lambda \in \C$. We insert the result of Lemma \ref{lem:Eisensteinprojection}, and obtain
\begin{align} \label{eq:specdefprojected}
\sum_{h \pmod*{2p}} I_h^{\reg}(4p\tau,1) = \frac{\pi \sqrt{p}(p+1)}{2} \CT_{s=1} \left[\frac{2^s\zeta^*(s)}{\zeta^*(2s-1)} \Fc_{\frac{1}{2},4p}^+\left(\tau,\frac{s}{2}+\frac{1}{4}\right)\right] + \lambda \theta(\tau).
\end{align}
Arguing as before shows that the projection of the theta lift to the plus space is of at most moderate growth towards $\af$ too.

Secondly, we establish that $F$ is of exponential decay as $v \to \infty$. To this end, we compare the constant terms in the involved Fourier expansions about $i\infty$. Note that $\Gamma_0(p)$ has the two inequivalent cusps $i\infty$ and $0$.
\begin{enumerate}[label={\rm (\arabic*)}]
\item Using equation \eqref{eq:BFIvol}, we have
\begin{align*}
\mathrm{vol}\left(\Gamma_0(p)\backslash\Hb\right) = - \frac{1}{6} (p+1).
\end{align*}
By virtue of Lemma \ref{lem:cuspparameters}, we deduce that $k_{\infty} = k_{0} = 0$ in Lemma \ref{lem:bfimain2} and we obtain
\begin{multline*}
\hspace*{\leftmargini} \sum_{\af \in \Gamma \backslash \Iso(V)} \widehat{\varepsilon}_{\af}  \left[\log\left(4 \beta_{\af}^2 \pi v\right) + \gamma + \psi\left(\frac{k_{\af}}{\beta_{\af}}\right) + \psi\left(1-\frac{k_{\af}}{\beta_{\af}}\right)\right] \\
= p\left[\log\left(\frac{4\pi v}{p^2}\right) + \log(4\pi v) - 2\gamma\right].
\end{multline*}
Consequently, the theta lift projected to the plus space has the constant term
\begin{multline*}
\hspace*{\leftmargini} \sum_{h \pmod*{2p}} I_h^{\reg}(4p\tau,1) \\
= -2\left(-\frac{1}{6} (p+1) \right)\sqrt{4pv} - \frac{\sqrt{p}}{2\pi} \left(\log\left(\frac{16\pi v}{p}\right) + \log(16\pi pv) - 2\gamma\right) + \ldots,
\end{multline*}
which simplifies to
\begin{align*}
\frac{2}{3}(p+1)p^{\frac{1}{2}} v^{\frac{1}{2}} - \frac{\sqrt{p}}{\pi} \left(\log(16 v) + \log(\pi) - \gamma\right).
\end{align*}

\item According to Lemma \ref{lem:GcExpansion}, we have the constant term
\begin{align*}
\Gc(\tau) = 
\frac{2}{3} v^{\frac{1}{2}} - \frac{1}{2\pi (p+1)}\log(16v) + \frac{2}{3} (1-i) \pi \cf(0) + \ldots,
\end{align*}
and by Theorem \ref{thm:cfsquareindices} (ii), we have
\begin{align*}
\frac{2}{3} (1-i) \pi \cf(0) = \frac{2}{\pi(p+1)}\left(\gamma - \log(2) - \frac{\zeta'(2)}{\zeta(2)} - \frac{p^2\log(p)}{p^2-1}\right).
\end{align*}

\item Rewriting the constant terms in equation \eqref{eq:DITconstantterms}, we infer
\begin{align*}
\Ec(\tau) = \frac{\sqrt{v}}{3} - \frac{1}{4\pi} \log(16v) + \frac{1}{\pi} \left(\gamma - \log(2) - \frac{\zeta'(2)}{\zeta(2)}\right) + \ldots.
\end{align*}
\end{enumerate}

Combining, we verify that
\begin{itemize}
\item the contributions to $v^{\frac{1}{2}}$ satisfy
\begin{align*}
\frac{2}{p-1}\frac{1}{3} + \frac{2}{3} &= \frac{\sqrt{p}}{p^2-1} \frac{2}{3}(p+1)p^{\frac{1}{2}},
\end{align*}
\item the contributions to $\log(16v)$ satisfy
\begin{align*}
\frac{2}{p-1}\left(-\frac{1}{4\pi}\right) + \left(-\frac{1}{2\pi (p+1)}\right) &= \frac{\sqrt{p}}{p^2-1}\left(- \frac{\sqrt{p}}{\pi}\right),
\end{align*}
\item the contributions to $1$ satisfy
\begin{multline*}
\hspace*{\leftmargini} \frac{2}{p-1}\frac{1}{\pi} \left(\gamma - \log(2) - \frac{\zeta'(2)}{\zeta(2)}\right) 
+ \frac{2}{\pi(p+1)}\left(\gamma - \log(2) - \frac{\zeta'(2)}{\zeta(2)} - \frac{p^2\log(p)}{p^2-1}\right) \\
= \Cf(p) +  \frac{\sqrt{p}}{p^2-1} \left(- \frac{\sqrt{p}}{\pi} \left(\log(\pi) - \gamma\right)\right).
\end{multline*}
\end{itemize}
This proves
\begin{multline*}
\frac{2}{p-1}\left(\frac{1}{3}v^{\frac{1}{2}} - \frac{1}{4\pi} \log(16v) + \frac{1}{\pi} \left(\gamma - \log(2) - \frac{\zeta'(2)}{\zeta(2)}\right)\right) \\
+ \frac{2}{3} v^{\frac{1}{2}} - \frac{1}{2\pi (p+1)}\log(16v) + \frac{2}{\pi(p+1)}\left(\gamma - \log(2) - \frac{\zeta'(2)}{\zeta(2)} - \frac{p^2\log(p)}{p^2-1}\right) \\
= \Cf(p) + \frac{\sqrt{p}}{p^2-1} \left(\frac{2}{3}(p+1)p^{\frac{1}{2}} v^{\frac{1}{2}} - \frac{\sqrt{p}}{\pi} \left(\log(16 v) + \log(\pi) - \gamma\right)\right),
\end{multline*}
with
\begin{multline*}
\Cf(p) = \frac{2}{\pi(p-1)} \left(\gamma - \log(2) - \frac{\zeta'(2)}{\zeta(2)}\right) + \frac{2}{\pi(p+1)} \left(\gamma - \log(2) - \frac{\zeta'(2)}{\zeta(2)} - \frac{p^2\log(p)}{p^2-1}\right) \\
+ \frac{\sqrt{p}}{p^2-1} \frac{\sqrt{p}}{\pi} \left(\log(\pi)-\gamma\right).
\end{multline*}
We simplify the expression for $\Cf(p)$ and $F$ is of exponential decay as $v \to \infty$.
\end{proof}

\section{The proof of Theorem \ref{thm:millsonthetalift}} \label{sec:ProofOfThm1.2}

\begin{proof}[Proof of Theorem \ref{thm:millsonthetalift}]
According to \cite{brufuim}*{Remark 5.3}, we have
\begin{align*}
\varphi_{\text{KM}}^{V^{-}} = -\frac{1}{\pi} \xi_{\frac{1}{2}} \varphi_0 \dm\mu(z).
\end{align*}
Recall from \cite{dit11annals}*{p.\ 973} that
\begin{align*}
\xi_{\frac{1}{2}} \Ec(\tau) = -2\Hc(\tau),
\end{align*}
and from \cite{bemo}*{Theorem 1.3 (iii)} (see equation \eqref{eq:GcShadow} too) that
\begin{align*}
\xi_{\frac{1}{2}} \Gc(\tau) = \frac{4}{p+1} \Big(\Hs_{1,1}(\tau)-\Hc(\tau)\Big) - \frac{4}{p} \Hs_{1,p}(\tau) - \frac{4}{1-p} \Hs_{p,p}(\tau).
\end{align*}
Combining with Theorem \ref{thm:GcAsThetaLift} (ii), we infer
\begin{multline*}
-4\pi p \frac{\sqrt{p}}{p^2-1} \sum_{h \pmod*{2p}} J_h^{\reg}(4p\tau,1) = \xi_{\frac{1}{2}} \left(\frac{2}{p-1}\Ec(\tau) + \Gc(\tau) - \Cf(p) \theta(\tau)\right) \\
= -2 \frac{2}{p-1} \Hc(\tau) + \frac{4}{p+1} \Big(\Hs_{1,1}(\tau)-\Hc(\tau)\Big) - \frac{4}{p} \Hs_{1,p}(\tau) - \frac{4}{1-p} \Hs_{p,p}(\tau).
\end{multline*}
Simplifying yields
\begin{align*}
\sum_{h \pmod*{2p}} J_h^{\reg}(4p\tau,1) = \frac{2}{\pi \sqrt{p}} \Hc(\tau) - \frac{p-1}{\pi p^{\frac{3}{2}}} \Hs_{1,1}(\tau) + \frac{p^2-1}{\pi p^{\frac{5}{2}}} \Hs_{1,p}(\tau) - \frac{p+1}{\pi p^{\frac{3}{2}}}\Hs_{p,p}(\tau)
\end{align*}
and applying the second assertion of \cite{bemo}*{Theorem 1.1} completes the proof.
\end{proof}

\section{The proof of Theorem \ref{thm:fouriercomparisonmain}} \label{sec:ProofOfThm1.4}

\begin{thm} \label{thm:fouriercomparisondetail}
Let $p$ be an odd prime.
We write $n = tm^2$ with $t$ fundamental.
\begin{enumerate}[label={\rm (\roman*)}]
\item Let $H_{\ell,p}(n)$ be Pei and Wang's \cite{peiwang} generalized Hurwitz class numbers (see equations \eqref{eq:HellNdef}, \eqref{eq:HNNdef}). If $n < 0$ with $n \equiv 0,1 \pmod*{4}$, then we have
\begin{align*}
\hspace*{\leftmargini} \sum_{Q \in \mathcal{Q}_{p,n} \slash \Gamma_0(p)} \frac{1}{\vt{\Gamma_0(p)_Q}} = \frac{4(p+1)}{p} H_{1,p}(-n) - \frac{2(p+1)}{p-1} H_{p,p}(-n).
\end{align*}
\item We define the local factors
\begin{align*}
\widetilde{A(2,n)} \coloneqq \begin{cases} 1 - 2^{-\frac{\nu_2(n)+3}{2}} & \text{if } 2 \nmid \nu_2(n), \\
1 - 2^{-\frac{\nu_2(n)}{2}} & \text{if } 2 \mid \nu_2(n), \frac{n}{2^{\nu_2(n)}} \equiv 3 \pmod*{4}, \\
1 & \text{if } 2 \mid \nu_2(n), \frac{n}{2^{\nu_2(n)}} \equiv 1 \pmod*{8}, \\
1-\frac{2}{3} \cdot 2^{-\frac{\nu_2(n)}{2}} & \text{if } 2 \mid \nu_2(n), \frac{n}{2^{\nu_2(n)}} \equiv 5 \pmod*{8},
\end{cases}
\end{align*}
and 
\begin{align*}
A(p,n) = \begin{cases}
\frac{1}{p}-(p+1)p^{-\frac{\nu_p(n)+3}{2}} & \text{if } 2 \nmid \nu_p(n), \\
\frac{1}{p} & \text{if } 2 \mid \nu_p(n), \left(\frac{\frac{n}{p^{\nu_p(n)}}}{p}\right) = 1, \\
\frac{1}{p} - 2p^{-\frac{\nu_p(n)}{2}-1} & \text{if } 2 \mid \nu_p(n), \left(\frac{\frac{n}{p^{\nu_p(n)}}}{p}\right) = -1.
\end{cases}
\end{align*}
for $p > 2$ from \cite{bemo}*{Section 3}. Suppose that $0 < n \equiv 0,1 \pmod*{4}$. Let $\widetilde{\varepsilon}_t > 1$ be the fundamental unit of $\Q\left(\sqrt{t}\right)$ and $h(t)$ be the usual class number in the wide sense. If $t \neq 1$, then we have
\begin{multline*}
\hspace*{\leftmargini} \frac{1}{2} \sum_{Q \in \mathcal{Q}_{p,n} \slash \Gamma_0(p)} \int_{\Gamma_0(p)_Q \backslash S_Q} \frac{\sqrt{n}\dm\tau}{Q(\tau,1)} = 2\pi\frac{p+1}{\sqrt{p}} \frac{h^*(n)}{\sqrt{n}}   \\
+ 8p^{\frac{3}{2}} \left(1-\chi_t(2)2^{-1}\right) \left(1-\chi_t(p)p^{-1}\right) T_{4p,0}^{\chi_t}(m) \widetilde{A(2,n)} A\left(p,tm^2\right) \frac{\log\left(\widetilde{\varepsilon}_t\right)}{\sqrt{t}} h(t).
\end{multline*}

\item If $0 < n = m^2 \equiv 0,1 \pmod*{4}$ is a square, then we have
\begin{multline*}
\hspace*{\leftmargini}  \frac{\sqrt{p}}{2\left(p^2-1\right)} \sum_{Q \in \mathcal{Q}_{p,m^2} \slash \Gamma_0(p)} \int_{\Gamma_0(p)_Q \backslash S_Q} \frac{\dm\tau}{Q(\tau,1)} \\
= \frac{1}{3(p-1)} \left(-5\gamma - \log(\pi) + 4\frac{\zeta'(2)}{\zeta(2)} + 4\log(2) - \frac{6\log(p)}{(p+1)} - 4 \tf_{1}(m)\right) \\
+ \frac{2}{p+1} \left[\left(2^{-\nu_2(m)}-1\right) \log(2) + \frac{p+1}{p-1}p^{-\nu_p(m)} \log(p) + \log(m) - \tf_{4p}(m) \right] 
\end{multline*}
\end{enumerate}
\end{thm}

\begin{proof}
Due to the projection to the plus space as in equation \eqref{eq:plusspaceprojection}, we have to insert $4p\tau$ in Lemma \ref{lem:bfimain2} and map $n \mapsto \frac{n}{4p}$ subsequently.
\begin{enumerate}[label={\rm (\roman*)}]
\item From the non-holomorphic parts, we get the identity
\begin{multline*}
\hspace*{\leftmargini} \frac{\sqrt{p}}{p^2-1} \sum_{h \pmod*{2p}} \deg Z\left(\frac{n}{4p},h\right) \frac{\erfc\left(2\sqrt{\pi \vt{\frac{n}{4p}}(4pv)}\right)}{2\sqrt{ \vt{\frac{n}{4p}}}} \\
=  \frac{2}{3} (1-i) \pi^{\frac{1}{2}} \cf(n) \Gamma\left(\frac{1}{2},4\pi\vt{n}v\right) 
+ \frac{2}{p-1} \frac{H_{1,1}(\vt{n})}{\sqrt{\vt{n}}} \frac{1}{\sqrt{\pi}} \Gamma \left( \frac{1}{2}, 4 \pi \vt{n} v \right)
\end{multline*}
for every $n < 0$ such that $n \equiv 0,1 \pmod*{4}$. Using 
\begin{align*}
\frac{\erfc(2\sqrt{\pi \vt{m} v}) }{2\sqrt{\vt{m}}} = \frac{\Gamma\left(\frac{1}{2}, 4\pi \vt{m} v\right)}{2\sqrt{\pi \vt{m}}},
\end{align*}
this simplifies to
\begin{align*}
\hspace*{\leftmargini} \frac{\sqrt{p}}{p^2-1} \sum_{h \pmod*{2p}} \deg Z\left(\frac{n}{4p},h\right) \frac{1}{2 \sqrt{  \pi \vt{\frac{n}{4p}}}}
=  \frac{2}{3} (1-i) \pi^{\frac{1}{2}} \cf(n) 
+ \frac{2}{p-1} \frac{1}{\sqrt{\pi}} \frac{H_{1,1}(\vt{n})}{\sqrt{\vt{n}}}
\end{align*}
Using equations \eqref{eq:quadraticformsidentification} and \eqref{eq:trimagdef}, this implies that
\begin{align*}
\hspace*{\leftmargini} \sum_{h \pmod*{2p}} \deg Z\left(\frac{n}{4p},h\right) &= \deg \sum_{h \pmod*{2p}} \sum_{Q \in \mathcal{Q}_{p,4p \frac{n}{4p},h} \slash \Gamma_0(p)} \frac{\tau_Q}{\vt{\Gamma_0(p)_Q}} \\
&= \sum_{Q \in \mathcal{Q}_{p,n} \slash \Gamma_0(p)} \frac{1}{\vt{\Gamma_0(p)_Q}}.
\end{align*}
The definition of $\cf(n)$ from equation \eqref{eq:c(n)def} gives
\begin{align*}
\cf(n) = \Ks_{\frac{1}{2},4p}^+\left(0,n;\frac{3}{2}\right),
\end{align*}
We use the first displayed equation on \cite{bemo}*{p.\ 28}. Namely, it states that if $1 \leq n \equiv 0,3 \pmod*{4}$ then
\begin{align} \label{eq:negativecoefficients}
4\pi(1+i)\overline{\Ks_{\frac{1}{2},4p}^+\left(0,-n;\frac{3}{2}\right)} n^{\frac{1}{2}} = \frac{12}{p} \left(H_{1,p}(n) + \frac{p}{1-p}H_{p,p}(n) \right).
\end{align}
Conjugating and mapping $n \mapsto -n$ yields
\begin{align*}
\Ks_{\frac{1}{2},4p}^+\left(0,n;\frac{3}{2}\right) = \frac{12}{p} \frac{1}{4\pi(1-i)n^{\frac{1}{2}}} \left(H_{1,p}(-n) + \frac{p}{1-p}H_{p,p}(-n) \right).
\end{align*}
Hence, we obtain
\begin{align*}
\hspace*{\leftmargini} \frac{p}{p^2-1} \sum_{Q \in \mathcal{Q}_{p,n} \slash \Gamma_0(p)} \frac{1}{\vt{\Gamma_0(p)_Q}} = \frac{2}{p}\left(H_{1,p}(-n) + \frac{p}{1-p}H_{p,p}(-n) \right) + \frac{2}{p-1} H_{1,1}(-n).
\end{align*}
According to the second part of \cite{bemo}*{Theorem 1.1}, we have
\begin{align*}
\frac{1}{1-p}H_{p,p}(-n) = H_{1,1}(-n) - \frac{p+1}{p}H_{1,p}(-n),
\end{align*}
and thus (i) follows.

\item Suppose that $0 < n \equiv 0,1 \pmod*{4}$ is not a square. Using equation \eqref{eq:quadraticformsidentification}, we obtain
\begin{align*}
\hspace*{\leftmargini} \frac{\sqrt{p}}{p^2-1} \sum_{h \pmod*{2p}} \sum_{Q \in \mathcal{Q}_{p,4p \frac{n}{4p},h} \slash \Gamma_0(p)} \mathrm{length}\left(\Gamma_0(p)_Q\backslash S_Q\right)
= \frac{2}{3}(1-i) \pi \cf(n) + \frac{2}{p-1} \frac{h^*(n)}{\sqrt{n}}.
\end{align*}
The definition of $\cf(n)$ from equation \eqref{eq:c(n)def} gives
\begin{align*}
\cf(n) = \Ks_{\frac{1}{2},4p}^+\left(0,n;\frac{3}{2}\right),
\end{align*}
for $n \neq \square$, and utilizing \cite{bemo}*{Proposition 3.4 (ii)} yields
\begin{align*}
\hspace*{\leftmargini} \cf(n) = \frac{L_{4p}\left(1,\chi_t\right)}{L_{4p}\left(2,\id\right)} T_{4p,0}^{\chi_t}(m) \left(A\left(2,tm^2\right)+\frac{1+i}{8}\right)A\left(p,tm^2\right),
\end{align*}
where $A(2,n)$ is as defined in \cite{bemo}*{Proposition 3.4 (ii)}.
Furthermore, we employ Dirichlet's class number formula, namely
\begin{align*}
L\left(1, \chi_t\right) = \frac{2h(t) \ln\left(\widetilde{\varepsilon}_t\right)}{\sqrt{t}},
\end{align*}
getting 
\begin{align*}
\frac{L_{4p}\left(1,\chi_t\right)}{L_{4p}(2, \id)}
&= \frac{L\left(1,\chi_t\right)}{\zeta(2)} \frac{1-\chi_t(2)2^{-1}}{1-2^{-2}} \frac{1-\chi_t(p)p^{-1}}{1-p^{-2}} \\
&= \frac{16p^2}{\pi^2\left(p^2-1\right)} \left(1-\chi_t(2)2^{-1}\right) \left(1-\chi_t(p)p^{-1}\right) \frac{\log\left(\widetilde{\varepsilon}_t\right)}{\sqrt{t}}h(t).
\end{align*} 
Noting that 
\begin{align*}
A(2,n)+\frac{1+i}{8} = 3 \cdot \frac{1+i}{8} \cdot \widetilde{A(2,n)}
\end{align*}
and that the length is normalized by $\frac{1}{2\pi}$ for positive indices (see equation \eqref{eq:trrealdef}) yields part (ii).

\item Suppose that $0 < n \equiv 0,1 \pmod*{4}$ is a square. Write $n = m^2$. Using Lemmas \ref{lem:cuspparameters} and \ref{lem:specialfunctionrelation}, we obtain ($m \mapsto \frac{m}{2p}$)
\begin{multline*}
\hspace*{\leftmargini} 2 \Cf(p) + \frac{\sqrt{p}}{p^2-1} \sum_{h \pmod*{2p}} \left[\tr_{\frac{n}{4p},h}(1) + \frac{p}{\sqrt{p}\pi} \sum_{\af \in \{0,\infty\}}  b_{\af}\left(\frac{n}{4p},h\right) \frac{\alpha\left(4 \frac{n}{4p} (4pv)\right)}{2} \right] \\
= \frac{2}{p-1} \left(\bff(n) + \frac{2}{4\pi}\alpha\left(4nv\right)\right)
+ \frac{1}{\pi(p+1)} \left(\gamma + \log\left(\pi n\right) +  \alpha\left(4nv\right) \right) + \frac{2}{3}(1-i) \pi \cf(n),
\end{multline*}
where we recall that the special function $\alpha$ is renormalized by $\frac{1}{4\pi}$ in \cite{dit11annals}*{Theorem 4} and \cite{alansa}*{(1.7)}. Since $0 < n \equiv 0,1 \pmod*{4}$ is a square, we deduce from the last paragraph of \cite{brufuim}*{Section 2} that
\begin{align*}
\sum_{h \pmod*{2p}} \sum_{\af \in \{0,\infty\}}  b_{\af}\left(\frac{n}{4p},h\right) = 4.
\end{align*}
Hence, the special function $\alpha$ cancels on both sides, and we infer
\begin{multline*}
\hspace*{\leftmargini}  \frac{\sqrt{p}}{p^2-1} \sum_{h \pmod*{2p}} \tr_{\frac{n}{4p},h}(1) \\
= \frac{2}{p-1} \bff(n) + \frac{1}{\pi(p+1)} \left(\gamma + \log\left(\pi n\right) \right) + \frac{2}{3}(1-i) \pi \cf(n) - 2 \Cf(p),
\end{multline*}
We insert the results of Theorem \ref{thm:cfsquareindices} (i), (iii), and the definition of $\Cf(p)$ from equation \eqref{eq:frakCdef} (Theorem \ref{thm:GcAsThetaLift} (ii)), which yields 
\begin{multline*}
\hspace*{\leftmargini}  \frac{\sqrt{p}}{p^2-1} \sum_{h \pmod*{2p}} \tr_{\frac{n}{4p},h}(1) = \frac{2}{p-1} \frac{2}{3\pi} \left(\gamma - 2\frac{\zeta'(2)}{\zeta(2)} - \log(4) + \frac{1}{2}\log(\pi) - \tf_{1}(m)\right) \\
+ \frac{1}{\pi(p+1)} \left(\gamma + \log\left(\pi n\right) \right) 
+ \frac{2}{\pi(p+1)} \left[ \gamma - 2 \frac{\zeta'(2)}{\zeta(2)} +
\log(2) \left(2^{-\nu_2(m)}-3\right) + \right. \\ \left.
\log(p) \left(\frac{1}{p+1} + \frac{p+1}{p-1}p^{-\nu_p(m)} - \frac{2p}{p-1} \right) - \tf_{4p}(m) \right] \\
- 2 \frac{4p}{\pi\left(p^2-1\right)} \left[\gamma - \log(2) - \frac{\zeta'(2)}{\zeta(2)} - \frac{p\log(p)}{2(p+1)} + \frac{1}{4}\left(\log(\pi)-\gamma\right) \right].
\end{multline*}
Lastly, we invoke equations \eqref{eq:quadraticformsidentification} and \eqref{eq:trrealdef} once more. Combining expressions yields part (iii). \qedhere
\end{enumerate}
\end{proof}

\section{The proof of Theorem \ref{thm:specdefcor}} \label{sec:ProofOfThm1.5}

We recall from Lemma \ref{lem:spectraldeformations} that
\begin{align*}
f_s(\tau) = \frac{\pi(p+1)}{6}\frac{1}{\zeta^*(2s-1)} \Fc_{0,p}(\tau,s)
\end{align*}
is a weight $0$ and level $p$ spectral deformation of $f = 1$, and we defined $B_{\af}(s)$ as the coefficient of $v^{1-s}$ in the Fourier expansion of $f_s$ about the cusp $\af$ in equation \eqref{eq:Bafdef}. To make Corollary \ref{cor:bfiCor6.3} explicit, we require the following small lemma.
\begin{lemma} \label{lem:levelreduction}
Let $N$ be square-free, $a \geq 1$, and $\chi_{\ell}(a)^2$ be the principal character of modulus $\ell$. Then, we have
\begin{align*}
\sum_{\ell \mid N} \mu(\ell) \chi_{\ell}(a)^2 = \begin{cases}
1 & \text{if } N \mid a, \\
0 & \text{if } N \nmid a.
\end{cases}
\end{align*}
\end{lemma}

\begin{rmk}
This lemma and a different proof can be found in \cite{mmrw}*{Lemma 7.1} as well.
\end{rmk}

\begin{proof}[Proof of Lemma \ref{lem:levelreduction}]
We decompose $N = N_a N'$, where $N_a$ divides all sufficiently high powers of $a$ and $\gcd(N', a) = 1$. Thus, we infer
\begin{align*}
\sum_{\ell \mid N} \mu(\ell) \chi_{\ell}(a)^2 &= \sum_{\substack{ \ell \mid N \\ \gcd(\ell, a) =1}} \mu(\ell) = \sum_{\ell \mid N'} \mu(\ell) = \begin{cases} 1 & N' = 1, \\
0 & N' > 1,
\end{cases}
\end{align*}
by M\"obius inversion. The result follows from observing that $N' = 1$ precisely when $N = N_a$, and since we are assuming $N$ is square-free, this is equivalent to $N$ dividing $a$.
\end{proof}

Now, we are in position to prove Theorem \ref{thm:specdefcor}.
\begin{proof}[Proof of Theorem \ref{thm:specdefcor}]
Recall that $b_{\af}(m,h)$ is the $(m,h)$-th Fourier coefficient of $\tilde\Theta_{K_\af}(\tau)$. Set
\begin{align*}
\lambda_{\af} \coloneqq \sum_{h \pmod*{2p}} b_{\af} (0,h).
\end{align*}
Combining Lemma \ref{lem:Eisensteinprojection} and Corollary \ref{cor:bfiCor6.3} as in the proof of Theorem \ref{thm:GcAsThetaLift} (ii) yields
\begin{multline*} 
\frac{1}{\sqrt{p}}\sum_{h \pmod*{2p}} I_h^{\reg}(4p\tau,1) = \frac{\pi (p+1)}{2} \CT_{s=1} \left[\frac{4^{\frac{s}{2}}\zeta^*(s)}{\zeta^*(2s-1)} \Fc_{\frac{1}{2},4p}^+\left(\tau,\frac{s}{2}+\frac{1}{4}\right) \right] \\
- \left(\frac{\partial}{\partial s}B_{\infty}(s)\Big\vert_{s=1} \cdot \lambda_{\infty}\theta(\tau) + \frac{\partial}{\partial s}B_{0}(s)\Big\vert_{s=1} \cdot \lambda_{0}\theta(\tau)\right),
\end{multline*}
compare equation \eqref{eq:specdefprojected}. On \cite{brufuim}*{p.\ 90}, it is stated that $b_{\af}(0,h) = 1$ if and only if $\tilde{h} = 0$. Otherwise, $b_{\af}(0,h) = 0$. Hence, we have
\begin{align*}
\lambda_{0} = \lambda_{\infty} = 1,
\end{align*}
and it suffices to determine $B_{\af}(s)$ for $\af \in \{0, \infty\}$. 
\begin{enumerate}
\item Lemma \ref{lem:Eisensteinexpansions} (i) gives
\begin{align*}
\Fc_{0,p}(\tau,s) = v^{s} +  \frac{4^{1-s}\pi \Gamma(2s-1)}{\Gamma\left(s\right)^2} \Ks_{0,p}(0,0;2s) v^{1-s} + \ldots,
\end{align*}
and Lemma \ref{lem:levelreduction} implies that
\begin{align*}
\Ks_{0,p}(0,0;2s) = \frac{\zeta(2s-1)}{\zeta(2s)} \left(1-\frac{1-p^{-(2s-1)}}{1-p^{-2s}}\right) = \frac{\zeta(2s-1)}{\zeta(2s)} \frac{p-1}{p^{2s}-1}.
\end{align*}
Thus, the spectral deformations $f_s$ from Lemma \ref{lem:spectraldeformations} have the constant terms
\begin{align*}
f_s(\tau) = \frac{\pi^{s+\frac{1}{2}}(p+1)}{6\Gamma\left(s-\frac{1}{2}\right)\zeta(2s-1)} v^{s} + \frac{p^2-1}{p^{2s}-1} \frac{\pi^{s+1}}{6 \Gamma(s)\zeta(2s)} v^{1-s} + \ldots
\end{align*}
about $i\infty$. Consequently, we have
\begin{multline*}
\hspace*{\leftmargini} B_{\infty}(s) = \frac{p^2-1}{p^{2s}-1} \frac{\pi^{s+1}}{6 \Gamma(s)\zeta(2s)} = 1 + \frac{1}{\pi^2\left(p^2-1\right)} \\
+ \pi^2\left[\left(p^2-1\right)\left(\gamma+\log(\pi)-2\frac{\zeta'(2)}{\zeta(2)}\right) - 2p^2\log(p)\right](s-1) + O\left((s-1)^2\right),
\end{multline*}
from which we read off $\frac{\partial}{\partial s}B_{\infty}(s)\big\vert_{s=1}$ directly.

\item Lemma \ref{lem:Eisensteinexpansions} (ii) gives
\begin{align*}
\left(\Fc_{0,p} \big\vert_0 W_p\right) (\tau,s) = \frac{4^{1-s} \pi \Gamma(2s-1)}{\Gamma(s)^2 p^s} \widetilde{\Ks}_{0,p}(0,0;2s) v^{1-s} + \ldots,
\end{align*}
and a quick computation shows (or see \cite{peiwangbook}*{Proposition 2.1} for example)
\begin{align*}
\widetilde{\Ks}_{0,p}(0,0;2s) = \frac{\zeta(2s-1)}{\zeta(2s)} \frac{1-p^{-(2s-1)}}{1-p^{-2s}} = \frac{\zeta(2s-1)}{\zeta(2s)} \frac{p^{2s}-p}{p^{2s}-1}.
\end{align*}
Thus, the spectral deformations $f_s$ from Lemma \ref{lem:spectraldeformations} have the constant terms
\begin{align*}
\left(f_s\big\vert_0 W_p\right)(\tau) = \frac{(p+1) \left(p^{2s}-p\right)\pi^{s+1}}{6\left(p^{2s}-1\right)\Gamma(s)\zeta(2s)} v^{1-s} + \ldots
\end{align*}
about $0$. Consequently, we have
\begin{multline*}
\hspace*{\leftmargini} B_{0}(s) = \frac{(p+1) \left(p^{2s}-p\right)\pi^{s+1}}{6\left(p^{2s}-1\right)\Gamma(s)\zeta(2s)} = p + \frac{1}{\pi^2\left(p^2-1\right)} \\
+ \pi^2\left[p\left(p^2-1\right)\left(\gamma+\log(\pi)-2\frac{\zeta'(2)}{\zeta(2)}\right) + 2p^2\log(p)\right](s-1) + O\left((s-1)^2\right),
\end{multline*}
from which we read off $\frac{\partial}{\partial s}B_{0}(s)\big\vert_{s=1}$ directly.
\end{enumerate}

Using Theorem \ref{thm:GcAsThetaLift} (ii), we arrive at
\begin{multline*}
\frac{1}{p}\left(\frac{2}{p-1}\Ec(\tau) + \Gc(\tau) - \Cf(p) \theta(\tau)\right)
= \frac{\pi}{2(p-1)} \CT_{s=1} \left[\frac{4^{\frac{s}{2}}\zeta^*(s)}{\zeta^*(2s-1)} \Fc_{\frac{1}{2},4p}^+\left(\tau,\frac{s}{2}+\frac{1}{4}\right) \right] \\
- \pi^2(p+1)\left(\gamma+\log(\pi)-2\frac{\zeta'(2)}{\zeta(2)}\right) \theta(\tau).
\end{multline*}
Rearranging completes the proof.
\end{proof}


\begin{bibsection}
\begin{biblist}
\bib{ahland}{article}{
   author={Ahlgren, S.},
   author={Andersen, N.},
   title={Algebraic and transcendental formulas for the smallest parts
   function},
   journal={Adv. Math.},
   volume={289},
   date={2016},
   pages={411--437},
}

\bib{alansa}{article}{
   author={Ahlgren, S.},
   author={Andersen, N.},
   author={Samart, D.},
   title={A polyharmonic Maass form of depth $3/2$ for $\slz$},
   journal={J. Math. Anal. Appl.},
   volume={468},
   date={2018},
   number={2},
   pages={1018--1042},
}

\bib{alfesthesis}{thesis}{
   author={Alfes, C.},
   title={CM values and Fourier coefficients of harmonic Maass forms},
   type={Ph.D. Thesis},
   organization={TU Darmstadt},
   date={2015},
}

\bib{air}{webpage}{
   author={Alfes-Neumann, C.},
   author={Burban, I.},
   author={Raum, M.},
   title={A classification of polyharmonic Maa{\ss} forms via quiver representations},
   url={https://arxiv.org/abs/2207.02278},
   year={2022},
   note={preprint},
}

\bib{alfehl}{article}{
   author={Alfes, C.},
   author={Ehlen, S.},
   title={Twisted traces of CM values of weak Maass forms},
   journal={J. Number Theory},
   volume={133},
   date={2013},
   number={6},
   pages={1827--1845},
}

\bib{anbs}{article}{
   author={Alfes-Neumann, C.},
   author={Bringmann, K.},
   author={Males, J.},
   author={Schwagenscheidt, M.},
   title={Cycle integrals of meromorphic modular forms and coefficients of
   harmonic Maass forms},
   journal={J. Math. Anal. Appl.},
   volume={497},
   date={2021},
   number={2},
   pages={Paper No. 124898, 15},
}

\bib{abs}{webpage}{
   author={Allen, M.},
   author={Beckwith, O.},
   author={Sharma, V.},
   title={Holomorphic projection for sesquiharmonic Maass forms},
   url={https://arxiv.org/abs/2411.05972},
   year={2024},
   note={preprint},
}

\bib{an15}{article}{
   author={Andersen, N.},
   title={Periods of the $j$-function along infinite geodesics and mock
   modular forms},
   journal={Bull. Lond. Math. Soc.},
   volume={47},
   date={2015},
   number={3},
   pages={407--417},
}

\bib{andu}{article}{
   author={Andersen, N.},
   author={Duke, W.},
   title={Modular invariants for real quadratic fields and Kloosterman sums},
   journal={Algebra Number Theory},
   volume={14},
   date={2020},
   number={6},
   pages={1537--1575},
}

\bib{anlagrho}{article}{
   author={Andersen, N.},
   author={Lagarias, J.\ C.},
   author={Rhoades, R.\ C.},
   title={Shifted polyharmonic Maass forms for ${\text {\rm PSL}}_2(\mathbb{Z})$},
   journal={Acta Arith.},
   volume={185},
   date={2018},
   number={1},
   pages={39--79},
}

\bib{andrews}{article}{
   author={Andrews, G. E.},
   title={The number of smallest parts in the partitions of $n$},
   journal={J. Reine Angew. Math.},
   volume={624},
   date={2008},
   pages={133--142},
}

\bib{bemo}{webpage}{
   author={Beckwith, O.},
   author={Mono, A.},
   title={A modular framework of generalized Hurwitz class numbers},
   url={https://arxiv.org/abs/2403.17829},
   year={2024},
   note={preprint},
}

\bib{bor98}{article}{
	author={Borcherds, R.},
	title={Automorphic forms with singularities on Grassmannians},
	journal={Invent. Math.},
	volume={132},
	date={1998},
	number={3},
	pages={491--562},
}

\bib{bor00}{article}{
   author={Borcherds, R.\ E.},
   title={Reflection groups of Lorentzian lattices},
   journal={Duke Math. J.},
   volume={104},
   date={2000},
   number={2},
   pages={319--366},
}

\bib{brdira}{article}{
   author={Bringmann, K.},
   author={Diamantis, N.},
   author={Raum, M.},
   title={Mock period functions, sesquiharmonic Maass forms, and non-critical values of $L$-functions},
   journal={Adv. Math.},
   volume={233},
   date={2013},
   pages={115--134},
}

\bib{thebook}{book}{
   author={Bringmann, K.},
   author={Folsom, A.},
   author={Ono, K.},
   author={Rolen, L.},
   title={Harmonic Maass forms and mock modular forms: theory and applications},
   series={American Mathematical Society Colloquium Publications},
   volume={64},
   publisher={American Mathematical Society, Providence, RI},
   date={2017},
   pages={xv+391},
}

\bib{brika20}{article}{
   author={Bringmann, K.},
   author={Kane, B.},
   title={An extension of Rohrlich's theorem to the $j$-function},
   journal={Forum Math. Sigma},
   volume={8},
   date={2020},
   pages={Paper No. e3, 33},
}

\bib{bruinierhabil}{book}{
	author={Bruinier, J.\ H.},
	title={Borcherds products on O(2, $l$) and Chern classes of Heegner
		divisors},
	series={Lecture Notes in Mathematics},
	volume={1780},
	publisher={Springer-Verlag, Berlin},
	date={2002},
}

\bib{BEY}{article}{
   author={Bruinier, J.\ H.},
   author={Ehlen, S.},
   author={Yang, T.},
   title={CM values of higher automorphic Green functions for orthogonal
   groups},
   journal={Invent. Math.},
   volume={225},
   date={2021},
   number={3},
   pages={693--785},
}

\bib{brufu02}{article}{
   author={Bruinier, J.\ H.},
   author={Funke, J.},
   title={On two geometric theta lifts},
   journal={Duke Math. J.},
   volume={125},
   date={2004},
   number={1},
   pages={45--90},
}

\bib{brufu06}{article}{
   author={Bruinier, J.\ H.},
   author={Funke, J.},
   title={Traces of CM values of modular functions},
   journal={J. Reine Angew. Math.},
   volume={594},
   date={2006},
   pages={1--33},
}

\bib{brufuim}{article}{
   author={Bruinier, J.\ H.},
   author={Funke, J.},
   author={Imamo\={g}lu, \"{O}.},
   title={Regularized theta liftings and periods of modular functions},
   journal={J. Reine Angew. Math.},
   volume={703},
   date={2015},
   pages={43--93},
}

\bib{bruku}{article}{
   author={Bruinier, J.\ H.},
   author={Kuss, M.},
   title={Eisenstein series attached to lattices and modular forms on
   orthogonal groups},
   journal={Manuscripta Math.},
   volume={106},
   date={2001},
   number={4},
   pages={443--459},
}

\bib{bruono13}{article}{
   author={Bruinier, J.\ H.},
   author={Ono, K.},
   title={Algebraic formulas for the coefficients of half-integral weight
   harmonic weak Maass forms},
   journal={Adv. Math.},
   volume={246},
   date={2013},
   pages={198--219},
}

\bib{BO}{article}{
   author={Bruinier, J.\ H.},
   author={Ono, K.},
   title={Heegner divisors, $L$-functions and harmonic weak Maass forms},
   journal={Ann. of Math. (2)},
   volume={172},
   date={2010},
   number={3},
   pages={2135--2181},
}

\bib{BS}{article}{
	title={Theta lifts for Lorentzian lattices and coefficients of mock theta functions},
	journal={Math. Z.},
	publisher={Springer Science and Business Media LLC},
	author={Bruinier, J.\ H.},
	author={Schwagenscheidt, M.},
	year={2020},
}

\bib{BY}{article}{
	author={Bruinier, J.\ H.},
   author={Yang, T.},
   title={Faltings heights of CM cycles and derivatives of $L$-functions},
   journal={Invent. Math.},
   volume={177},
   date={2009},
   number={3},
   pages={631--681},
}

\bib{cohen75}{article}{
   author={Cohen, H.},
   title={Sums involving the values at negative integers of $L$-functions of quadratic characters},
   journal={Math. Ann.},
   volume={217},
   date={1975},
   number={3},
   pages={271--285},
}

\bib{crfu}{article}{
   author={Crawford, J.},
   author={Funke, J.},
   title={The Shimura-Shintani correspondence via singular theta lifts and
   currents},
   journal={Int. J. Number Theory},
   volume={19},
   date={2023},
   number={10},
   pages={2383--2417},
}

\bib{dit11annals}{article}{
   author={Duke, W.},
   author={Imamo\={g}lu, \"{O}.},
   author={T\'{o}th, \'{A}.},
   title={Cycle integrals of the $j$-function and mock modular forms},
   journal={Ann. of Math. (2)},
   volume={173},
   date={2011},
   number={2},
   pages={947--981},
}

\bib{ditimrn}{article}{
   author={Duke, W.},
   author={Imamo\=glu, \"O.},
   author={T\'oth, \'A.},
   title={Real quadratic analogs of traces of singular moduli},
   journal={Int. Math. Res. Not. IMRN},
   date={2011},
   number={13},
   pages={3082--3094},
}

\bib{eiza}{book}{
   author={Eichler, M.},
   author={Zagier, D.},
   title={The theory of Jacobi forms},
   series={Progress in Mathematics},
   volume={55},
   publisher={Birkh\"{a}user Boston, Inc., Boston, MA},
   date={1985},
   pages={v+148},
}


\bib{funke}{article}{
   author={Funke, Jens},
   title={Heegner divisors and nonholomorphic modular forms},
   journal={Compositio Math.},
   volume={133},
   date={2002},
   number={3},
   pages={289--321},
}

\bib{garthwaite}{article}{
   author={Garthwaite, S.\ A.},
   title={Vector-valued Maass-Poincar\'{e} series},
   journal={Proc. Amer. Math. Soc.},
   volume={136},
   date={2008},
   number={2},
   pages={427--436},
}

\bib{goldstein}{article}{
   author={Goldstein, L.\ J.},
   title={Dedekind sums for a Fuchsian group. I},
   journal={Nagoya Math. J.},
   volume={50},
   date={1973},
   pages={21--47},
}

\bib{Gross-Kohnen-Zagier}{article}{
   author={Gross, B.},
   author={Kohnen, W.},
   author={Zagier, D.},
   title={Heegner points and derivatives of $L$-series. II},
   journal={Math. Ann.},
   volume={278},
   date={1987},
   number={1-4},
   pages={497--562},
}

\bib{hamo}{article}{
   author={Harvey, J. A.},
   author={Moore, G.},
   title={Algebras, BPS states, and strings},
   journal={Nuclear Phys. B},
   volume={463},
   date={1996},
   number={2-3},
   pages={315--368},
}

\bib{hoevel}{thesis}{
   author={Hövel, M.},
   title={Automorphe Formen mit Singularitäten auf dem hyperbolischen Raum},
   type={Ph.D. Thesis},
   organization={TU Darmstadt},
   date={2012},
}

\bib{jkk1}{article}{
   author={Jeon, D.},
   author={Kang, S.-Y.},
   author={Kim, C. H.},
   title={Weak Maass-Poincar\'{e} series and weight 3/2 mock modular forms},
   journal={J. Number Theory},
   volume={133},
   date={2013},
   number={8},
   pages={2567--2587},
}

\bib{jkk14}{article}{
   author={Jeon, D.},
   author={Kang, S.-Y.},
   author={Kim, C.\ H.},
   title={Cycle integrals of a sesqui-harmonic Maass form of weight zero},
   journal={J. Number Theory},
   volume={141},
   date={2014},
   pages={92--108},
}

\bib{jkk24}{webpage}{
   author={Jeon, D.},
   author={Kang, S.-Y.},
   author={Kim, C.\ H.},
   title={Hecke Equivariance of Divisor Lifting with respect to Sesquiharmonic Maass Forms},
   year={2024},
   url={https://arxiv.org/abs/2407.21447},
   note={preprint},
}

\bib{jkkm24}{webpage}{
   author={Jeon, D.},
   author={Kang, S.-Y.},
   author={Kim, C.\ H.},
   author={Matsusaka, T.},
   title={A unified approach to Rohrlich-type divisor sums},
   year={2024},
   url={https://arxiv.org/abs/2410.12571},
   note={preprint},
}


\bib{katoksarnak}{article}{
   author={Katok, S.},
   author={Sarnak, P.},
   title={Heegner points, cycles and Maass forms},
   journal={Israel J. Math.},
   volume={84},
   date={1993},
   number={1-2},
   pages={193--227},
}

\bib{kohnen85}{article}{
	AUTHOR = {Kohnen, W.},
	TITLE = {Fourier coefficients of modular forms of half-integral weight},
	JOURNAL = {Math. Ann.},
	VOLUME = {271},
	YEAR = {1985},
	NUMBER = {2},
	PAGES = {237--268},
}

\bib{koza81}{article}{
   author={Kohnen, W.},
   author={Zagier, D.},
   title={Values of $L$-series of modular forms at the center of the
   critical strip},
   journal={Invent. Math.},
   volume={64},
   date={1981},
   number={2},
   pages={175--198},
}

\bib{koza84}{article}{
   author={Kohnen, W.},
   author={Zagier, D.},
   title={Modular forms with rational periods},
   conference={
      title={Modular forms},
      address={Durham},
      date={1983},
   },
   book={
      series={Ellis Horwood Ser. Math. Appl.: Statist. Oper. Res.},
      publisher={Horwood, Chichester},
   },
   date={1984},
   pages={197--249},
}

\bib{kubota}{book}{
   author={Kubota, T.},
   title={Elementary theory of Eisenstein series},
   publisher={Kodansha, Ltd., Tokyo; Halsted Press [John Wiley \& Sons,
   Inc.], New York-London-Sydney},
   date={1973},
   pages={xi+110},
}

\bib{kumi}{article}{
   author={Kudla, S. S.},
   author={Millson, J. J.},
   title={The theta correspondence and harmonic forms. I},
   journal={Math. Ann.},
   volume={274},
   date={1986},
   number={3},
   pages={353--378},
}

\bib{kumi2}{article}{
   author={Kudla, S. S.},
   author={Millson, J. J.},
   title={The theta correspondence and harmonic forms. II},
   journal={Math. Ann.},
   volume={277},
   date={1987},
   number={2},
   pages={267--314},
}

\bib{lagrho}{article}{
   author={Lagarias, J.\ C.},
   author={Rhoades, R.\ C.},
   title={Polyharmonic Maass forms for ${\text {\rm PSL}}_2(\mathbb{Z})$},
   journal={Ramanujan J.},
   volume={41},
   date={2016},
   number={1-3},
   pages={191--232},
}

\bib{luozhou}{webpage}{
   author={Luo, Y.},
   author={Zhou, H.},
   title={The classification and representations of positive definite ternary quadratic forms of level $4N$},
   year={2024},
   url={https://arxiv.org/abs/2402.17443},
   note={preprint},
}

\bib{maass64}{book}{
   author={Maass, H.},
   title={Lectures on modular functions of one complex variable},
   series={Tata Institute of Fundamental Research Lectures on Mathematics
   and Physics},
   volume={29},
   edition={2},
   note={With notes by Sunder Lal},
   publisher={Tata Institute of Fundamental Research, Bombay},
   date={1983},
   pages={iii+262},
}

\bib{males}{article}{
   author={Males, J.},
   title={Higher Siegel theta lifts on Lorentzian lattices, harmonic Maass
   forms, and Eichler-Selberg type relations},
   journal={Math. Z.},
   volume={301},
   date={2022},
   number={4},
   pages={3555--3569},
}

\bib{mamo}{article}{
   author={Males, J.},
   author={Mono, A.},
   title={Local Maass forms and Eichler-Selberg relations for
   negative-weight vector-valued mock modular forms},
   journal={Pacific J. Math.},
   volume={322},
   date={2023},
   number={2},
   pages={381--406},
}

\bib{mmrw}{webpage}{
   author={Males, J.},
   author={Mono, A.},
   author={Rolen, L.},
   author={Wagner, I.},
   title={Central $L$-values of newforms and local polynomials},
   year={2023},
   url={https://arxiv.org/abs/2306.15519},
   note={preprint},
}

\bib{mat20}{article}{
   author={Matsusaka, T.},
   title={Polyharmonic weak Maass forms of higher depth for $\slz$},
   journal={Ramanujan J.},
   volume={51},
   date={2020},
   number={1},
   pages={19--42},
}

\bib{mat19}{article}{
   author={Matsusaka, T.},
   title={Traces of CM values and cycle integrals of polyharmonic Maass forms},
   journal={Res. Number Theory},
   volume={5},
   date={2019},
   number={1},
   pages={Paper No. 8, 25},
}

\bib{niwa}{article}{
   author={Niwa, S.},
   title={Modular forms of half integral weight and the integral of certain
   theta-functions},
   journal={Nagoya Math. J.},
   volume={56},
   date={1975},
   pages={147--161},
}

\bib{peiwang}{article}{
   author={Pei, D.},
   author={Wang, X.},
   title={A generalization of Cohen-Eisenstein series and Shimura liftings and some congruences between cusp forms and Eisenstein series},
   journal={Abh. Math. Sem. Univ. Hamburg},
   volume={73},
   date={2003},
   pages={99--130},
}

\bib{peiwangbook}{book}{
   author={Pei, D.},
   author={Wang, X.},
   title={Modular forms with integral and half-integral weights},
   publisher={Science Press Beijing, Beijing; Springer, Heidelberg},
   date={2012},
   pages={x+432},
}

\bib{schw18}{thesis}{
   author={Schwagenscheidt, M.},
   title={Regularized Theta Lifts of Harmonic Maass Forms},
   type={Ph.D. Thesis},
   organization={TU Darmstadt},
   date={2018},
}

\bib{scheit}{article}{
   author={Scheithauer, N.},
   title={Generalized Kac-Moody algebras, automorphic forms and Conway's
   group. I},
   journal={Adv. Math.},
   volume={183},
   date={2004},
   number={2},
   pages={240--270},
}

\bib{sest}{article}{
   author={Serre, J.-P.},
   author={Stark, H. M.},
   title={Modular forms of weight $1/2$},
   conference={
      title={Modular functions of one variable, VI},
      address={Proc. Second Internat. Conf., Univ. Bonn, Bonn},
      date={1976},
   },
   book={
      series={Lecture Notes in Math.},
      volume={Vol. 627},
      publisher={Springer, Berlin-New York},
   },
   date={1977},
   pages={27--67},
}

\bib{shim}{article}{
   author={Shimura, G.},
   title={On modular forms of half integral weight},
   journal={Ann. of Math. (2)},
   volume={97},
   date={1973},
   pages={440--481},
}

\bib{shin}{article}{
   author={Shintani, T.},
   title={On construction of holomorphic cusp forms of half integral weight},
   journal={Nagoya Math. J.},
   volume={58},
   date={1975},
}

\bib{siegel56}{article}{
   author={Siegel, C.\ L.},
   title={Die Funktionalgleichungen einiger Dirichletscher Reihen},
   journal={Math. Z.},
   volume={63},
   date={1956},
   pages={363--373},
}

\bib{vass}{thesis}{
   author={Vassileva, I.\ N.},
   title={Dedekind eta function, Kronecker limit formula and Dedekind sum for the Hecke group},
   type={Ph.D. Thesis},
   organization={University of Massachusetts Amherst},
   date={1996},
}

\bib{wald1}{article}{
   author={Waldspurger, J.-L.},
   title={Correspondance de Shimura},
   language={French},
   journal={J. Math. Pures Appl. (9)},
   volume={59},
   date={1980},
   number={1},
   pages={1--132},
}

\bib{wald2}{article}{
   author={Waldspurger, J.-L.},
   title={Sur les coefficients de Fourier des formes modulaires de poids
   demi-entier},
   language={French},
   journal={J. Math. Pures Appl. (9)},
   volume={60},
   date={1981},
   number={4},
   pages={375--484},
}

\bib{zagier75}{article}{
   author={Zagier, D.},
   title={Nombres de classes et formes modulaires de poids $3/2$},
   language={French, with English summary},
   journal={C. R. Acad. Sci. Paris S{\'e}r. A-B},
   volume={281},
   date={1975},
   number={21},
   pages={Ai, A883--A886},
}

\bib{zagier76}{article}{
   author={Zagier, D.},
   title={On the values at negative integers of the zeta-function of a real quadratic field},
   journal={Enseign. Math. (2)},
   volume={22},
   date={1976},
   number={1-2},
   pages={55--95},
}

\bib{zagier02}{article}{
   author={Zagier, D.},
   title={Traces of singular moduli},
   conference={
      title={Motives, polylogarithms and Hodge theory, Part I},
      address={Irvine, CA},
      date={1998},
   },
   book={
      series={Int. Press Lect. Ser.},
      volume={3, I},
      publisher={Int. Press, Somerville, MA},
   },
   date={2002},
   pages={211--244},
}


\end{biblist}
\end{bibsection}
\end{document}